\documentclass{article}
\usepackage{amsmath, amssymb, amsfonts, amsthm, mathrsfs, bbold, indentfirst}
\usepackage{xcolor}
\usepackage[normalem]{ulem} 
\usepackage{soul} 
\usepackage[breakable, theorems, skins]{tcolorbox}

\newtheorem{definition}{\bf Definition}[section] 
\newtheorem{theorem}{\bf Theorem}[section] 
\newtheorem{corollary}{\bf Corollary}[section] 
\newtheorem{lemma}{\bf Lemma}[section] 
\newtheorem{proposition}{\bf Proposition}[section] 
\newtheorem{remark}{\bf Remark}[section]

\begin{document}
	
	\title{Asymptotic Equivalence and Decay Characterization in Second-Grade Fluid Equations}
	\author{Felipe W. Cruz\footnotemark[1] \,\footnotemark[2], \,
		C\'esar J. Niche\footnotemark[3], \,
        Cilon F. Perusato\footnotemark[4] \,\footnotemark[5], \,\\
		Marko A. Rojas-Medar\footnotemark[6]}
	\date{}
	\maketitle
	
	\renewcommand{\thefootnote}{\fnsymbol{footnote}}
	
	\footnotetext[1]{Departamento de Ecuaciones Diferenciales y An\'alisis Num\'erico, 
		Facultad de Matem\'aticas, 
		Universidad de Sevilla, 
		41012 Seville, Spain.}
    \footnotetext[2]{Departamento de Matem\'atica, 
		Centro de Ci\^encias Exatas e da Natureza, 
		Universidade Federal de Pernambuco, 
		Recife - PE, CEP 50740-560, Brazil ({\tt felipe.wcruz@ufpe.br}).}
	\footnotetext[3]{Departamento de Matem\'atica Aplicada, Instituto de Matem\'atica, Universidade Federal do Rio de Janeiro, Rio de Janeiro - RJ, CEP 21941-909, Brazil ({\tt cniche@im.ufrj.br}).}
   \footnotetext[4]{Institut Camille Jordan, 
		Universit\'e Lyon 1, 
		Villeurbanne, France ({\tt perusato@math.univ-lyon1.fr}).}
	\footnotetext[5]{Departamento de Matem\'atica, 
		Centro de Ci\^encias Exatas e da Natureza, 
		Universidade Federal de Pernambuco, 
		Recife - PE, CEP 50740-560, Brazil ({\tt cilon.perusato@ufpe.br}).}
    \footnotetext[6]{Departamento de Matem\'atica, Facultad de Ciencias, Universidad de Tarapac\'a, 1010069 Arica, Arica y Parinacota, Chile ({\tt mmedar@academicos.uta.cl}).}

	\renewcommand{\thefootnote}{\arabic{footnote}}
	
	\maketitle
	
	\begin{abstract}
		We establish new asymptotic results for the solutions of the second-grade fluids equations and characterize their decay rate in terms of the behavior of the initial data. Moreover, assuming more regularity for the initial data, we study the large-time behavior of these solutions by comparing them to the solutions of the linearized equations. As a consequence, we obtain lower bounds for the solutions. Other auxiliary results of interest are discussed.
	\end{abstract}

\section{Introduction}
Second-grade fluids are non-Newtonian fluids whose stress tensor is 

\begin{equation}
\label{eqn:tensor}
\sigma = -p \, Id + 2 \nu \, A_1 + \alpha_1 A_2 + \alpha_2 A_1 ^2,    
\end{equation}
where $A_1$ and $ A_2$  are the Rivlin - Ericksen  \cite{Rivlin-Ericksen-1955} tensors given by

\begin{equation}
\label{eqn:rivlin}
A_1 (u) = \frac{1}{2} \left(\nabla u + \nabla u ^t \right), \quad A_2 (u) = \partial_t A_1 + (u \cdot \nabla) A_1 + (\nabla u ) ^t A_1 + A_1 (\nabla u).
\end{equation}
Here, \(\mathbf{u}\) and \(p\) denote the velocity of the fluid and its pressure, respectively, defined on $\mathbb{R}^3$.

Some properties of solutions to systems of equations arising from \eqref{eqn:tensor} depend on the relation between the coefficients $\alpha_1$ and $\alpha_2$ in \eqref{eqn:rivlin}, see Dunn and Fosdick \cite{Dunn}, Fosdick and Ragajopal \cite{Fosdick, Rajagopal}. In particular, the choice $\alpha_1 + \alpha_2 = 0, \nu > 0$  leads to the thermodynamically consistent  equations

\begin{equation}\label{eqn:1}
		\begin{cases}
    \partial_t (\mathbf{u} - \alpha \Delta \mathbf{u}) - \mu \Delta \mathbf{u} + \operatorname{curl}(\mathbf{u} - \alpha \Delta \mathbf{u}) \times \mathbf{u} + \nabla p = {\boldsymbol 0}, \\ \noalign{\smallskip}
    \operatorname{div} \mathbf{u} = 0, \\ \noalign{\smallskip}
    \mathbf{u}(0) = \mathbf{u}_0,
\end{cases} 
	\end{equation}
where $\alpha = \alpha_1$ and \(\mu > 0\) is the kinematic viscosity, see Dunn and Fosdick \cite{Dunn}. Due to its importance for modeling different non-Newtonian fluids and its relation to other well known models, different aspects concerning solutions to \eqref{eqn:1} have been studied, see Bresch and Lemoine \cite{Bresch-Lemoine-1998}, Busuioc \cite{Busuioc02,Busuioc18}, Busuioc and Ratiu \cite{Busuioc-Ratiu03}, Cioranescu and El Hac\`ene \cite{elhacene}, Cioranescu, Girault and Ragajopal \cite{book_cioranescu}, Cioranescu and Girault \cite{vivette_Doina}, Clark, Friz and Rojas-Medar \cite{Clark}, Coscia and Galdi \cite{coscia}, Friz, Guill\'en-Gonz\'alez and Rojas-Medar \cite{Friz1, Friz2}, Galdi, Grobbelaar-Van Dalsen and Sauer \cite{Dalsen1, Dalsen2}, Galdi and Sequeira \cite{Sequeira}, Girault and Scott \cite{Girault-Scott99}, Iftimie \cite{Dragos02}, Linshiz and Titi \cite{Linshiz-Titi10}, Paicu and Vicol \cite{Paicu-Vicol11},  Paicu and Raugel \cite{PaicuRaugel},  Paicu, Yu and Zhu \cite{Paicu-Yu-Zhu24,Paicu-Yu-ZhuSIMA24}, le Roux \cite{Roux} and references therein. 

The existence of the regularizing term $- \alpha \:\!\partial_{t} \Delta \mathbf{u}$ in \eqref{eqn:1} implies (formally) the energy identity
\begin{equation}\label{int_energy_equa}
\frac{d}{dt} \!\left(\bigl\|\mathbf{u}(t)\bigr\|_{L^2 (\mathbb{R}^3) ^3}^2 + \alpha \:\!\bigl\|\nabla \:\!\mathbf{u}(t)\bigr\|_{L^2 (\mathbb{R}^3) ^3}^2\right) = -2 \mu \:\!\bigl\|\nabla \:\!\mathbf{u} (t)\bigr\|_{L^2 (\mathbb{R}^3) ^3}^2,
\end{equation}
provided $\mathbf{u}$ is a sufficiently regular solution. This naturally leads to the question of what is the decay rate for the natural $H^{1}_{\alpha}(\mathbb{R}^3) ^3$ norm  

\begin{equation}
\label{eqn:norma-h1}
\bigl\|\mathbf{u}(t)\bigr\|_{H^1_\alpha (\mathbb{R}^3) ^3}^2 := \bigl\|\mathbf{u}(t)\bigr\|_{L^2(\mathbb{R}^3) ^3}^2 + \alpha \:\!\bigl\|\nabla \:\!\mathbf{u}(t)\bigr\|_{L^2(\mathbb{R}^3) ^3}^{2},
\end{equation}
which is equivalent to the usual $H^1 (\mathbb{R}^3) ^3$ norm.

Our main goal in this article is to prove decay estimates  for this norm for solutions to \eqref{eqn:1} in terms of the initial data $u_0$. To accomplish this, we associate to each $u_0 \in L^2 (\mathbb{R}^3)$ a decay character $r^{\ast} = r^{\ast} (u_0)$, which quantifies the ``order'' of $\widehat{u_0} (\xi)$ near the origin in the frequency domain, i.e.  $r^{\ast} = r^{\ast} (u_0)$ if $|\widehat{u_0} (\xi)| \approx |\xi|^{r^{\ast}}$ near $\xi = 0$. This idea has been developed by Bjorland and Schonbek \cite{Bjorland-Schonbek09}, Brandolese \cite{BrandoSIMA2016} and Niche and Schonbek \cite{cesar_london}. We adopt the definitions and conventions from Brandolese \cite{BrandoSIMA2016}, who introduced generalized decay characters $r^{\ast} _+$ and $r^{\ast} _-$ when aiming for upper and lower decay bounds respectively, and showed that when $r^{\ast} _+ = r^{\ast} _-$, they coincide with the original $r^{\ast}$, for full details see Section \ref{decay-character-section}.  

We state now our first main result. For the definition of $V_2 (\mathbb{R}^3) ^3 \subset H^3 (\mathbb{R}^3) ^3$, see Section \ref{existence-solutions}.
\newpage
\begin{theorem}\label{Main1-Intro}
		Let $\mathbf{u}_0 \in V_2 (\mathbb{R}^3) ^3$ and suppose that $r^{\ast} _+ (\mathbf{u}_0) = r^{\ast} _+$. Additionally, assume that $\Vert \mathbf{u}_0 \Vert _{V_2 (\mathbb{R}^3) ^3} < \epsilon$ for a sufficiently small $\epsilon > 0$. Then, for any weak solution 	$\mathbf{u}$  to \eqref{eqn:1}, the following estimate holds:
		\[
		\bigl\|\mathbf{u}(t)\bigr\|_{H^1_\alpha(\mathbb{R}^3) ^3}^2 \leq C\, (t + 1)^{-\min \left\{ \frac{3}{2} + r^{\ast}_+, \frac{5}{2} \right\}}, \quad \,\,\forall \; t \geq 0,
		\]
		where the constant $C>0$ depends only on $ \|\mathbf{u}_0\|_{V_2 (\mathbb{R}^3) ^3}, \alpha, r^{\ast} _+$, and $\mu$.
	\end{theorem}

There are few results concerning decay of solutions to \eqref{eqn:1}. Assuming small data in $H^3$, Galdi and Passerini \cite{Galdi-1995} proved that $\Vert u(t) \Vert _{L^2} ^2 \leq C \, \left( \ln (t) \right) ^{-1}$, for large enough $t$. Nečasová and Penel \cite{Necasova-Penel06} showed that, for initial data $\mathbf{u}_0 \in H^3 \cap L^1$ with $\operatorname{div} \mathbf{u}_0 = 0$ and no restrictions on size, solutions decay as

\begin{displaymath}
\Vert  \mathbf{u} (t) \Vert _{L^2} ^2 +  \Vert  \Delta \mathbf{u} (t) \Vert _{L^2} ^2   \leq C (1 + t) ^{- \frac{1}{2}}. 
\end{displaymath}
 In both of these results, the authors assume the existence of weak solutions. In Theorem \ref{Main1-Intro} we provide a decay with improved rates for the natural norm \eqref{eqn:norma-h1} for a much larger family of initial data, see Remark \ref{remark-large-family} and note that $r^{\ast} (\mathbf{u}_0) = 0$ for $\mathbf{u}_0 \in L^2 \cap L^1$. Moreover, as the linear part of the equation decays as $t^{- \left( \frac{3}{2} + r^{\ast}_+ \right)}$, see Theorem \ref{teodeclupl158}, we are able to separate the contribution of the linear and nonlinear parts of the solution to the decay in terms of initial data. For the sake of completeness,  in  Section \ref{existence-solutions} we provide a proof of the existence of weak solutions to \eqref{eqn:1}. The choice of subspace of $H^3$ and the hypotheses on small data in Theorem \ref{Main1-Intro} are due to them being essential in the proof of this existence result. For similar decay estimates for solutions to models related to \eqref{eqn:1}, see Anh and Trang \cite{Trang2018}, Bjorland \cite{BjorlandSIMA08}, Bjorland and Schonbek \cite{clayton}, Busuioc \cite{Busuioc09},  Gao, Lyu and Yao \cite{Gao-Lyu-Yao}, Zhao \cite{Zhao2021}. 

We now study some aspects of the asymptotic behavior of solutions to \eqref{eqn:1}, in particular we study decay of $\mathbf{u}(t) - \mathbf{\bar{u}} (t)$, where $\mathbf{\bar{u}}$ is a solution to the pseudo-parabolic linear part

\begin{equation}\label{lplppe}
		\begin{cases}
			\partial_t (\overline{\mathbf{u}} - \alpha \Delta \overline{\mathbf{u}}) - \mu \Delta \overline{\mathbf{u}} = {\boldsymbol 0}, \\ \noalign{\smallskip}
			\operatorname{div} \overline{\mathbf{u}} = 0, \\ \noalign{\smallskip}
			\overline{\mathbf{u}}(0) = \mathbf{u}_0 \in H^1_\alpha(\mathbb{R}^3) ^3.
		\end{cases} 
	\end{equation} 

Our second main result in this article is the following. 

\begin{theorem}\label{Main2-intro}
	Let $\mathbf{u}_0 \in H^4(\mathbb{R}^3)$ with $\operatorname{div} \mathbf{u}_0 = 0$, and suppose that $r^{\ast} _+ (\mathbf{u}_0) = r^{\ast} _+$ and that  $\mathbf{u}_0$ is small in $V_2 (\mathbb{R}^3)$, as in Theorem \ref{Main1-Intro}. Let \(\mathbf{u}\) be a weak solution to \eqref{eqn:1}, and let $\overline{\mathbf{u}}$ be the solution to the linear part \eqref{lplppe} with the same  initial data \(\mathbf{u}_0 \in H^4(\mathbb{R}^3)\). Then, any weak solution
	\(\mathbf{u}(t)\) satisfies the estimate
	\[
	\bigl\|\mathbf{u}(t) - \overline{\mathbf{u}}(t)\bigr\|_{H^1_\alpha(\mathbb{R}^3) ^3}^2 \leq C\, (t + 1)^{-\min \left\{ \frac{5}{2} + \frac{3}{2}r^{\ast} _+, \frac{5}{2} \right\}}, \quad \,\,\forall \; t \geq 0,
	\]
	where the constant $C>0$ depends only on \(  \|\mathbf{u}_0\|_{H^1_\alpha(\mathbb{R}^3) ^3}, \alpha, r^{\ast} _+\), and \(\mu \).

\end{theorem}
Then, the solution $\mathbf{u}$ of \eqref{eqn:1} is asymptotically equivalent to the solution $\overline{\mathbf{u}}$ of the pseudo-parabolic equation \eqref{plppe} with the same data. Other aspects of asymptotic behavior of \eqref{eqn:1}, in particular of the vorticity $\omega = \nabla \times \mathbf{u}$, have been studied by Jaffal-Mourtada \cite{Jaffal-Mourtada}. We need to increase the regularity of initial data to $\mathbf{u}_0 \in H^4(\mathbb{R}^3)$ in order to use the estimates for $\overline{\mathbf{u}}$ in Theorem \ref{teoder128}, see Remark \ref{remark-H4} for a detailed explanation of this fact.

As a consequence of Theorem \ref{Main2-intro}, lower bounds for decay of $\mathbf{u}$ can be obtained for a restricted range of $r^{\ast}$, providing sharp estimates for this family of initial data.  

\begin{corollary}
\label{lower-bound} Let $\mathbf{u}_0$ be as in Theorem \ref{Main2-intro}. If $r^* _+= r^* _+ (\mathbf{u}_0) \in (-3/2, 1)$, then the weak solution $\mathbf{u}$ to \eqref{eqn:1} which such initial data satisfies
\[
\bigl\|\mathbf{u}(t)\bigr\|_{H^1_\alpha(\mathbb{R}^3)}^2 \geq C\, (t + 1)^{-\min \left\{ \frac{3}{2} + r^*, \frac{5}{2} \right\}}, \quad \,\,\forall \; t \geq 0,
\]
where $C>0$ depends only on $\|\mathbf{u}_0\|_{H^1_\alpha(\mathbb{R}^3)}, \alpha, r^*, \mu $.
\end{corollary}

This article is organized as follows. In Section \ref{preliminaries} we recall definitions, results and estimates needed later on in our proofs. In Section \ref{proof-Main1} we prove Theorem \ref{Main1-Intro}, while in Section \ref{proof-Main2} we prove Theorem \ref{Main2-intro} and Corollary \ref{lower-bound}.

\paragraph{Acknowledgments.}
	The authors thank Drago\c{s} Iftimie for his comments, suggestions  and remarks concerning this work.  F.\;W.\;Cruz was partially supported by CNPq grants No.\;304605/2022-0 and No.\;200039/2024-5. C.\;F.\;Perusato was partially supported by a CNPq grant No. \;200124/2024-2 and CNPq grant bolsa PQ No.\; 310444/2022-5. M.\;A.\;Rojas-Medar was partially funded by the National Agency for Research and Development (ANID) through the Fondecyt project 1240152. 

\section{Preliminaries}

\label{preliminaries}
\subsection{Existence of weak solutions} \label{existence-solutions} We provide a sketch of a proof  of the existence of weak solutions for equations \eqref{eqn:1} in $\mathbb{R} ^3$. As we follow closely the arguments in Cioranescu and Girault \cite{vivette_Doina} for the existence of a solution in bounded domains, and in Bjorland and Schonbek \cite{clayton} for using such kind of results to obtain solutions in the whole space for the Camassa-Holm equations, we only indicate the main points in our proof and do not provide details for  the remaining, standard arguments.

We first recall the definitions and results we need from Cioranescu and Girault \cite{vivette_Doina}. Let $\Omega \subset \mathbb{R}^3$ be an open bounded domain with $\partial \Omega \in C^{3,1}$.  We then consider the following functions spaces

\begin{equation}\label{eqn:spaces}
    \begin{split}
        H^1 _0 (\Omega)  = \{v \in H^1 (\Omega):  v = 0  \, \,  \text{in}  \, \, \partial \Omega \}, \\ 
        V = \{\mathbf{v} \in H^1 _0 (\Omega) ^3: \operatorname{div} \, \mathbf{v} = 0 \, \,  \text{in} \, \, \partial \Omega \}, \notag \\  
        V_2 = \{\mathbf{v} \in V: \operatorname{curl} (\mathbf{v} - \alpha \Delta \mathbf{v}) \in L^2 (\Omega) ^3 \}.
    \end{split}
\end{equation} 
From Lemma 2.1 in \cite{vivette_Doina}, we have that the $H^3$ norm and the norm $\Vert \cdot \Vert _{V_2}$ induced by the inner product

\begin{equation}
\label{eqn:v2-norm}
    (\mathbf{u}, \mathbf{v}) _{V_2} = (\mathbf{u}, \mathbf{v}) + \alpha (\nabla \mathbf{u}, \nabla \mathbf{v}) + (\operatorname{curl} (\mathbf{u} - \alpha \Delta \mathbf{u}),\operatorname{curl}  (\mathbf{v} - \alpha \Delta \mathbf{v}))
\end{equation}
are equivalent in $V_2$, where $(\cdot,\cdot)$ is the usual $L^2$ inner product. 

We recall now the trilinear form

\begin{displaymath}
    b(\mathbf{u}, \mathbf{v}, \mathbf{w}) = \sum _{i,j = 1} ^3 = \int_{\Omega} u_j \frac{\partial v_i}{x_j} w _i \, dx. 
\end{displaymath}
The fact that

\begin{displaymath}
\int_{\Omega} \operatorname{curl}(\mathbf{u} - \alpha \Delta \mathbf{u}) \times \mathbf{u} \cdot \mathbf{v} \, dx = b(\mathbf{u}, \mathbf{u}, \mathbf{v}) - \alpha  b(\mathbf{u}, \Delta \mathbf{u}, \mathbf{v})  + \alpha b(\mathbf{v}, \Delta \mathbf{u}, \mathbf{u})    
\end{displaymath}
leads to the following definition.

\begin{definition} \label{solution-bounded-domain}
    Let $\mathbf{u}_0 \in V_2$. We say that $\mathbf{u} \in L^{\infty} (0,T; V_2)$ such that $\mathbf{u'} \in L^{\infty} (0,T; V)$ is a weak solution to \eqref{eqn:1} if for a.e. $t$ in $(0,T)$ and every $\mathbf{v} \in V$ 

\begin{displaymath}
    (\mathbf{u'}, \mathbf{v}) + \alpha \,  (\nabla \mathbf{u'}, \nabla \mathbf{v}) + \mu \,  (\nabla \mathbf{u}, \nabla \mathbf{v}) + b(\mathbf{u}, \mathbf{u}, \mathbf{v}) - \alpha \,   b(\mathbf{u}, \Delta \mathbf{u}, \mathbf{v})  + \alpha \,  b(\mathbf{v}, \Delta \mathbf{u}, \mathbf{u}) = 0.
\end{displaymath}
    
\end{definition}

The following is the key result  which we will use as main ingredient in our proof.

\begin{theorem}[Theorem 4.7, \cite{vivette_Doina}] \label{existence-cioranescu-girault} For any small enough $\mathbf{u}_0 \in V_2$ there exists a unique weak solution to \eqref{eqn:1} for all $t \geq 0$, in the sense of Definition \ref{solution-bounded-domain}.
\end{theorem}

We now proceed to sketch the proof of the existence of a weak solution to \eqref{eqn:1} in $\mathbb{R}^3$. Take $\mathbf{u}_0 \in V_2 ({\mathbb{R}^3})$, where

\begin{displaymath}
    V_2 ({\mathbb{R}^3}) = \{\mathbf{u}: \mathbf{u} \in H^1 (\mathbb{R}^3) ^3, \operatorname{curl} (\mathbf{v} - \alpha \Delta \mathbf{v}) \in L^2 (\mathbb{R}^3) ^3 \},
\end{displaymath}
endowed with the norm induced by \eqref{eqn:v2-norm}. Consider the ball $B^{R_k}$, centered at the origin with radius $R_k$. Localize $\mathbf{u}_0$ by taking $\mathbf{u}_0 ^k = \mathbf{u}_0 \, \chi _{B^{R_k - \epsilon}}$, for some small $\epsilon > 0$, where $\chi _{B^{R_k - \epsilon}}$ is a smooth cutoff function which is identically one in $B^{R_k - \epsilon}$ and equal to zero on $\partial B^{R_k}$.  As $\mathbf{u}_0 ^k \in V_2 = V_2 (B^{R_k})$, from Theorem \ref{existence-cioranescu-girault} we have a unique weak solution $\mathbf{u}^k$ in $B^{R_k}$. From the statement of Theorem 2.1 and the proof of Proposition 4.6 in \cite{vivette_Doina} it can be seen that $\mathbf{u}^k$ and $\mathbf{u'}^ {\, k}$ are bounded, in $L^{\infty} (\mathbb{R} _+; V_2)$ and $L^{\infty} (\mathbb{R} _+; V)$ respectively,  by constants that do not depend on the size of $B^{R_k}$. In fact, these constants depend only on $\alpha, \mu$, the constant in the Sobolev embedding  $H^3 \subset C^1 _0$ and the continuity constant in the Stokes operator, see pages 312 and 320, (2.13) and page 330 in \cite{vivette_Doina}. Moreover, from (4.18) in the statement of Theorem 4.6, smallness of initial data in  Theorem 2.1 can also be expressed in terms of these constants only, i.e. independently of $k$.

Now take $\mathbf{u}^k$ from the previous paragraph and extend it to $\mathbb{R} ^3$ by making it identically zero outside of $B^{R_k}$. By abuse of notation, call this extension $\mathbf{u}^k$ as well. It follows from the previous paragraph that these global-in-space $\mathbf{u}^k$ are uniformly bounded in $k$. Then we have 

\begin{align} 
& \mathbf{u} ^k  \overset{\ast}{\rightharpoonup} \mathbf{u} \, \, \text{in} \, \,  L^{\infty} (\mathbb{R} _+; V_2 ({\mathbb{R}^3})), \label{eqn:convergence-u} \\
& \mathbf{u'} ^k  \overset{\ast}{\rightharpoonup} \mathbf{u'} \, \, \text{in} \, \,  L^{\infty} (\mathbb{R} _+; V ({\mathbb{R}^3})).\label{eqn:convergence-u-prime}
\end{align}

We now provide the definition of weak solution for \eqref{eqn:1}.  

\begin{definition} \label{solution-whole-space}
    Let $\mathbf{u}_0 \in V_2 ({\mathbb{R}^3})$. We say that $\mathbf{u} \in L^{\infty} (0,T; V_2 ({\mathbb{R}^3}))$ such that $\mathbf{u'} \in L^{\infty} (0,T; V ({\mathbb{R}^3}))$ is a weak solution to \eqref{eqn:1} if for a.e. $t$ in $(0,T)$ and every compactly supported $\phi \in C^{\infty} _c ({\mathbb{R}^3})$ 
\begin{displaymath}
    (\mathbf{u'}, \phi) + \alpha \,  (\nabla \mathbf{u'}, \nabla \phi) + \mu \,  (\nabla \mathbf{u}, \nabla \phi) + b(\mathbf{u}, \mathbf{u}, \phi) - \alpha \,   b(\mathbf{u}, \Delta \mathbf{u}, \phi)  + \alpha \,  b(\phi, \Delta \mathbf{u}, \mathbf{u}) = 0.
\end{displaymath}   
    
\end{definition}
Now take an arbitrary $\phi \in C^{\infty} _c ({\mathbb{R}^3})$ and $k$ such that $\phi \subset B^{R_k}$. Extend $\phi$ to $B^{R_k}$ by defining it as zero in the complement of its support. Then $\phi \in V (B^{R_k})$ and we can apply Theorem  \ref{existence-cioranescu-girault} to obtain

\begin{align}
\label{eqn:approx-convergence}
    (\mathbf{u'}  ^{\, k}, \phi) & + \alpha \,   (\nabla \mathbf{u'} ^{\, k}, \nabla \mathbf{v}) + \mu \,  (\nabla \mathbf{u^k}, \nabla \phi) + b(\mathbf{u}^k, \mathbf{u}^k, \phi) \notag \\ & - \alpha \,   b(\mathbf{u}^k, \Delta \mathbf{u} ^k, \phi)  + \alpha \, b(\phi, \Delta \mathbf{u} ^k, \mathbf{u} ^k) = 0.
\end{align}
Convergence of $\mathbf{u} ^k$ and $\mathbf{u'} ^k$ in \eqref{eqn:convergence-u} and \eqref{eqn:convergence-u-prime} and standard arguments for treating each term in \eqref{eqn:approx-convergence} lead us to conclude the following.

\begin{theorem}For any small enough $\mathbf{u}_0 \in V_2 (\mathbb{R}^3)$ there exists a unique weak solution to \eqref{eqn:1} for all $t \geq 0$.
\end{theorem}

\subsection{Decay character theory}\label{decay-character-section} In this Section we provide definitions and results concerning the decay character and its relation to decay of solutions to linear dissipative  equations. Throughout this article, we denote the Fourier transform either by $\mathcal{F}$ or \,$\widehat{}$, 
	$$
	\mathcal{F}\{\varphi\}(\xi) \equiv \widehat{\varphi}(\xi) :=
	\int_{\mathbb{R}^n} \! e^{-i \,\xi \cdot x}
	\varphi(x) \,dx, \qquad i^2 = -1.
	$$
	In particular, by Plancherel's theorem,
	we have $\| \widehat{\varphi}\|_{L^2 (\mathbb{R}^n)} =
	(2 \pi)^{\frac{n}{2}} \| \varphi \|_{L^2 (\mathbb{R}^n)}$,
	for every $\varphi \in L^{2} (\mathbb{R}^n)$.

\begin{definition}[Bjorland and Schonbek \cite{Bjorland-Schonbek09}, Niche and Schonbek \cite{cesar_london}] \label{decay-indicator}
	Let $\mathbf{v}_0 \in L^2(\mathbb{R}^n)$ and $B(\rho) := \{ \xi \in \mathbb{R}^{n} \,;\, |\xi| \leq \rho \}$. For $r \in \left(-\frac{n}{2}, \infty\right)$, we define the \textit{decay indicator} of $\mathbf{v}_0$ as
	$$P_r(\mathbf{v}_0) :=  \lim_{\rho\rightarrow 0^+} \rho^{-2r-n}\int_{B(\rho)} |\widehat{\mathbf{v}_0}(\xi)|^2 \, d\xi ,$$
	provided this limit exists.
\end{definition}

\begin{remark}
	Note that the decay indicator is always zero when $r \leq -n/2$. Moreover, observe that the decay indicator compares $\left|\widehat{\mathbf{v}_0}(\xi)\right|^2$ to $|\xi|^{2r}$ near $\xi=0$. 
\end{remark}

\begin{definition}[Bjorland and Schonbek \cite{Bjorland-Schonbek09}, Niche and Schonbek \cite{cesar_london}] \label{df-decay-character} The \textit{decay character} of $\mathbf{v}_0 \in L^2(\mathbb{R}^n)$, denoted by $r^*=r^*(\mathbf{v}_0)$, is the unique $r \in \left(-\frac{n}{2}, \infty \right)$ such that $0 < P_r (\mathbf{v}_0)<\infty$, provided this number exists.
\end{definition}

\begin{theorem}[Niche and Schonbek \cite{cesar_london}] \label{cnmesl15}
Suppose $\mathbf{v}_0 \in H^s(\mathbb{R}^n)$ with $s > 0$. Let $r_s^*(\mathbf{v}_0) = r^*( \Lambda^s \mathbf{v}_0)$, where $\Lambda = \left( - \Delta \right) ^{\frac{1}{2}}$. Then $-\frac{n}{2} + s < r_s^*(\mathbf{v}_0) < \infty$ and $r_s^*(\mathbf{v}_0) = s + r^*(\mathbf{v}_0)$.
\end{theorem}

 \begin{remark} \label{remark-large-family}
 Note that the decay character can be explicitly computed in some important instances, for example $\mathbf{v}_0 \in L^2 \left( \mathbb{R}^n \right) \cap L^p \left( \mathbb{R}^n \right), 1 \leq p < 2$, see Ferreira, Niche and Planas \cite{Ferreira}. In these cases, we have $r^{\ast} (\mathbf{v}_0) = - n \left(1 - \frac{1}{p} \right)$.
 \end{remark}

In Definitions \ref{decay-indicator} and \ref{df-decay-character}, assumptions are made so as to ensure a positive $P_r (\mathbf{v}_0)$. However, these do not hold universally for all $\mathbf{v}_0 \in L^2 (\mathbb{R}^n)$. As demonstrated by Brandolese \cite{BrandoSIMA2016}, it is possible to construct initial data with strong oscillations near the origin such that the limit in Definition \ref{decay-indicator} fails to exist for certain values of $r$. Consequently, the decay character is not defined in such cases. He proposed a slightly more general definition of decay character than those in Definitions \ref{decay-indicator} and \ref{df-decay-character}, which we adopt in order to state Theorems \ref{Main1-Intro} and \ref{Main2-intro} and Corollary \ref{lower-bound}. 

\begin{definition}[Brandolese \cite{BrandoSIMA2016}]
	Let  $\mathbf{v}_0 \in L^2(\mathbb{R}^n)$ and $B_\rho=\{\xi\in \mathbb{R}^n\colon|\xi| \leq \rho\}$.
	The generalized lower and upper decay indicators of $v_0$ are
	\[
	P_r(\mathbf{v}_0)_- = \liminf_{\rho \to 0^+} \rho ^{-2r-n} \int _{B_\rho} |\widehat{\mathbf{v}_0} (\xi) |^2 \, d \xi
	\qquad
\] 
and
\[
	P_r(\mathbf{v}_0)_+ = \limsup_{\rho \to 0^+} \rho ^{-2r-n} \int _{B_\rho} |\widehat{\mathbf{v}_0} (\xi) |^2 \, d \xi,
	\]
respectively. When $P_r(\mathbf{v}_0)_-=P_r(\mathbf{v}_0)_+$, then we can define the {decay indicator} of~$\mathbf{v}_0$ as 
	$P_r(\mathbf{v}_0) = P_r(\mathbf{v}_0)_-=P_r(\mathbf{v}_0)_+$\,.
\end{definition}

\begin{definition}[Brandolese \cite{BrandoSIMA2016}]
\label{def-lorenzo}
	The {generalized upper and lower decay characters} of~$\mathbf{v}_0 \in L^2(\mathbb{R}^n)$ are respectively defined by
	\begin{align*}
		r^{\ast}_+ (\mathbf{v}_0) &=\sup\{r\in\mathbb{R}\colon P_r(\mathbf{v}_0)_+<\infty\},  \\
		r^{\ast} _- (\mathbf{v}_0) &=\inf\{r\in\mathbb{R}\colon P_r(\mathbf{v}_0)_->0\}.
	\end{align*}
\end{definition}

The decay character, as defined in Definition \ref{df-decay-character}, exists if and only if $r^{\ast}_+ (\mathbf{v}_0) = r^{\ast}_- (\mathbf{v}_0)$ in Definition \ref{def-lorenzo}. Alternatively, Brandolese \cite{BrandoSIMA2016} proved that the decay character $r^{\ast} (\mathbf{v}_0)$ exists if and only if $\mathbf{v}_0$ lies within a specific subset of a homogeneous Besov space, namely $\mathbf{v}_0 \in \dot{\mathcal{A}} ^{- \left( \frac{n}{2} + r^{\ast} \right)} _{2, \infty} \subset \dot{B} ^{- \left( \frac{n}{2} + r^{\ast} \right)} _{2, \infty}$. 

\subsection{The associated linear problem}

We shall consider here the following pseudo-parabolic equation in \(\mathbb{R}^3 \times (0, \infty)\):
\begin{equation}\label{plppe}
		\begin{cases}
    \partial_t (\overline{\mathbf{u}} - \alpha \Delta \overline{\mathbf{u}}) - \mu \Delta \overline{\mathbf{u}} = {\boldsymbol 0}, \\ \noalign{\smallskip}
    \operatorname{div} \overline{\mathbf{u}} = 0, \\ \noalign{\smallskip}
    \overline{\mathbf{u}}(0) = \mathbf{u}_0 \in H^1_\alpha(\mathbb{R}^3).
\end{cases} 
\end{equation}
A straightforward calculation gives the energy identity
\begin{equation}\label{energy_psudoLP15}
\frac{d}{dt} \bigl\|\overline{\mathbf{u}}(t)\bigr\|_{H^1_\alpha (\mathbb{R}^3)}^2  = -2 \mu \:\!\bigl\|\nabla \,\overline{\mathbf{u}} (t)\bigr\|_{L^2 (\mathbb{R}^3)}^2.    
\end{equation}
Since $\mathbf{u}_0 \in H^1_\alpha(\mathbb{R}^3)$, we have, in particular, that $\overline{\mathbf{u}} (t) \in L^2 (\mathbb{R}^3)$. Hence, the transform $\widehat{\overline{{\mathbf{u}}}}(t)$ is well defined. Taking the Fourier transform of \eqref{plppe}, we get 
\begin{equation*}
		\begin{cases}
(1 + \alpha |\xi|^2) \widehat{\overline{{\mathbf{u}}}}_t + \mu |\xi|^2 \widehat{\overline{{\mathbf{u}}}} = {\boldsymbol 0}, \\ \noalign{\smallskip}
i \xi \cdot \widehat{\overline{{\mathbf{u}}}}(\xi, t) = 0, \quad \xi \in \mathbb{R}^3, \quad t > 0, \\ \noalign{\smallskip}
\widehat{\overline{{\mathbf{u}}}}(0) = \widehat{\mathbf{u}_0}.
\end{cases}
\end{equation*}
So,
\[
\widehat{\overline{{\mathbf{u}}}}(\xi, t) = e^{-\frac{\mu |\xi|^2}{1 + \alpha |\xi|^2} t} \:\!\widehat{\overline{{\mathbf{u}}}}(\xi, 0) \equiv e^{t M_{\alpha, \mu}(\xi)} \:\!\widehat{\mathbf{u}_0}(\xi),
\]
where
\begin{equation}\label{eqn:deft_M}
M_{\alpha, \mu}(\xi):= - \frac{\mu|\xi|^2}{1+\alpha|\xi|^2}.    
\end{equation}
In particular:
\begin{equation}\label{8924lp}
    \bigl|\widehat{\overline{{\mathbf{u}}}}(\xi, t)\bigr| \leq \bigl|\widehat{\mathbf{u}_0}(\xi)\bigr|, \quad \,\,\forall \; \xi \in \mathbb{R}^3 \quad \text{and} \quad t \geq 0.
\end{equation}

Next, we characterize the decay of solutions to \eqref{plppe} in terms of the decay character $r^*(\mathbf{u}_0) = r^*$.

\begin{theorem}\label{teodeclupl158}
Let $\mathbf{u}_0 \in H_\alpha^1(\mathbb{R}^n)$, with $r^*(\mathbf{u}_0) = r^*$, that is, 
$$
	\label{myr}
	r^*(\mathbf{u}_0)=\max\{r\in\mathbb{R}\colon P_r(\mathbf{u}_0)_+<\infty\}=\min\{r\in\mathbb{R}\colon P_r(\mathbf{u}_0)_->0\}.
$$
If $-\frac{n}{2} < r^* < \infty$, then for the solution $\overline{{\mathbf{u}}}(t)$ of \eqref{plppe} with initial data $\mathbf{u}_0$ we have that there exist constants $C_1$, $C_2$, $C_3 > 0$ such that
\[
C_1(t+1)^{-\left(\frac{n}{2} + r^{*}\right)} \leq \bigl\|\overline{\mathbf{u}}(t)\bigr\|_{H^1_\alpha(\mathbb{R}^n)}^2 \leq C_2(t+C_3)^{-\left(\frac{n}{2} + r^{*}\right)}, \quad \,\,\forall \; t \geq 0.
\]
\end{theorem}
\begin{remark}
	The proof of this result follows the approach of C. Niche in \cite{Niche-Voigt}. Moreover, with minor adjustments to our argument (following the ideas of L. Brandolese in \cite{BrandoSIMA2016}), it is possible to establish the following result:
		\begin{itemize}
			\item[-]
			If $P_r(\mathbf{u}_0)_->0$ then there is a constant $C_1>0$ such that, for all $t>0$,
			\[
			C_1(1+t)^{-\frac{1}{\alpha}(r+n/2)}\le \|\mathbf{\bar{u}}(t)\|_2^2.
			\phantom{\le C_2(1+t)^{-\frac1\alpha(r+n/2)}}.
			\]
			\item[-]
			If $P_r(\mathbf{u}_0)_+<\infty$, then there is a constant $C_2>0$ such that, for all $t>0$,
			\[
			\phantom{C_1(1+t)^{-\frac{1}{\alpha}(r+n/2)}\le}
			\|\mathbf{\bar{u}}(t)\|_2^2\le C_2(1+t)^{-\frac1\alpha(r+n/2)}.
			\]
		\end{itemize}
\end{remark}	
\begin{proof} 
Initially, note, by Theorem \ref{cnmesl15}, that \( r^*(\nabla \mathbf{u}_0) = 1 + r^*(\mathbf{u}_0) \), with \( r^*(\mathbf{u}_0) = r^* \in \left(\frac{n}{2}, \infty\right) \). Since \( P_r(\mathbf{u}_0) > 0 \), it follows from the definition of the limit that there exist sufficiently small \(\rho_0 > 0\) and \(c_1 > 0\) such that for \(0<\rho < \rho_0\), 
\[
c_1 < \rho^{-2r-n} \int_{B(\rho)} |\widehat{\mathbf{u}_0}(\xi)|^2 \,d\xi, \quad \text{ and } \quad c_1 < \rho^{-2(r+1)-n} \int_{B(\rho)} |\xi|^2 |\widehat{\mathbf{u}_0}(\xi)|^2 \,d\xi.
\]
Note that $\rho_0$ and $c_1$ depend on $P_r(\mathbf{u}_0)$. From now on, we will denote \( B(t) := \left\{ \xi \in \mathbb{R}^n \,;\, |\xi| \leq \rho(t) \right\} \), for some continuous, positive and decreasing \( \rho(t) \) to be determined later. Thus:
\begin{eqnarray*}
&\displaystyle  (2\pi)^n \|\overline{\mathbf{u}}(t)\|_{H^1_\alpha(\mathbb{R}^n)}^2 \geq \int_{B(t)} (1 + \alpha |\xi|^2) e^{2t M_{\alpha,\mu}(\xi)} |\widehat{\mathbf{u}_0}(\xi)|^2 \,d\xi&   \\ \noalign{\smallskip}  &\displaystyle  \geq \int_{B(t)} (1 + \alpha |\xi|^2) e^{-2 \mu |\xi|^2 t} |\widehat{\mathbf{u}_0}(\xi)|^2 \,d\xi& \\ \noalign{\smallskip} &\displaystyle  \geq e^{-2\mu t \rho(t)^2} \int_{B(t)} |\widehat{\mathbf{u}_0}(\xi)|^2 \,d\xi + \alpha e^{-2\mu t \rho(t)^2} \int_{B(t)} |\xi|^2 |\widehat{\mathbf{u}_0}(\xi)|^2 \,d\xi& \\ \noalign{\smallskip} &\displaystyle  \geq c_1 \rho(t)^{2r+n} \:\!e^{-2\mu t \rho(t)^2} + \alpha \:\!c_1 \rho(t)^{2(r+1)+n} \:\!e^{-2\mu t \rho(t)^2}.&
\end{eqnarray*}
Choosing \(\rho(t) := \rho_0 (t + 1)^{-\frac{1}{2}}\), we obtain
\[
(2\pi)^n \|\overline{\mathbf{u}}(t)\|_{H^1_\alpha(\mathbb{R}^n)}^2 \geq c_1 e^{-2\mu \rho_0^2} [\rho_{0}^{2r+n} (t + 1)^{-(r+\frac{n}{2})} + \alpha \rho_{0}^{2(r+1)+n} (t + 1)^{-(r+1+\frac{n}{2})}],
\]
which shows that
\[
\|\overline{\mathbf{u}}(t)\|_{H^1_\alpha(\mathbb{R}^n)}^2 \geq C_1 (t + 1)^{-(r + \frac{n}{2})}, \quad \,\,\forall \; t \geq 0.
\]
Here, \( C_1 = C_1(\mu, \alpha, r, n, P_r(\mathbf{u}_0)) \in \mathbb{R}^+\). 

Now, to derive the upper estimate, we use the Schonbek's method, known in the literature as ``Fourier splitting'' technique (see \cite{schonbek1, schonbek3}). For this, we define the closed ball centered at the origin with radius \(\rho(t)\):
\[
B(t) := \left\{ \xi \in \mathbb{R}^n \,; \, |\xi| \leq \rho(t) \equiv \left( \frac{g'(t)}{2 \mu g(t) - \alpha g'(t)} \right)^{\frac{1}{2}} \right\}\!,
\]
where \(g\) is a differentiable function on \((0, \infty)\) satisfying
\[
g(0) = 1, \quad g(t) \geq 1, \quad g'(t) > 0, \quad \text{and} \quad 2 \mu g(t) > \alpha g'(t), \quad \,\,\forall \; t  > 0.
\]
Using Plancherel's theorem, it follows from the energy identity \eqref{energy_psudoLP15} that
$$
\frac{d}{dt} \!\left[\:\! \int_{\mathbb{R}^n} \left( 1 + \alpha |\xi|^2 \right) \!|\widehat{\overline{\mathbf{u}}}(\xi,t)|^2 \, d\xi \right] = - 2\mu \int_{\mathbb{R}^n} |\xi|^2 |\widehat{\overline{\mathbf{u}}}(\xi,t)|^2 \, d\xi.
$$
Multiplying the above identity by \(g(t)\), we obtain
\begin{eqnarray*}
    \frac{d}{dt} \!\left[ g(t) \int_{\mathbb{R}^n} (1 + \alpha |\xi|^2) |\widehat{\overline{\mathbf{u}}}|^2 \, d\xi \right] &=& g'(t) \int_{\mathbb{R}^n} (1 + \alpha |\xi|^2) |\widehat{\overline{\mathbf{u}}}|^2 \, d\xi \\ \noalign{\medskip} & & - \, 2 \mu g(t) \int_{\mathbb{R}^n} \frac{|\xi|^2 }{1 + \alpha |\xi|^2} (1 + \alpha |\xi|^2) |\widehat{\overline{\mathbf{u}}}|^2 \, d\xi.
\end{eqnarray*}
From the definition of \(B(t)\), we have
\[
\frac{2 \mu g(t) |\xi|^2}{1 + \alpha |\xi|^2} > g'(t), \quad \,\,\forall \; \xi \in B(t)^c \equiv \mathbb{R}^{n} \, \backslash \, B(t).
\]
Thus, by \eqref{8924lp}, we have
\begin{eqnarray*}
   & &  \hspace{-0.4cm} \frac{d}{dt} \!\left[ g(t) \int_{\mathbb{R}^n} (1 + \alpha |\xi|^2) |\widehat{\overline{\mathbf{u}}}(\xi, t)|^2 \, d\xi \right] \\ \noalign{\smallskip} & & \leq g'(t) \int_{\mathbb{R}^n} (1 + \alpha |\xi|^2) |\widehat{\overline{\mathbf{u}}}(\xi, t)|^2 \, d\xi  \\&& -  2 \mu g(t) \int_{B(t)^c} \frac{|\xi|^2}{1 + \alpha |\xi|^2} (1 + \alpha |\xi|^2) |\widehat{\overline{\mathbf{u}}}(\xi, t)|^2 \, d\xi \\  & & \leq g'(t) \int_{B(t)} (1 + \alpha |\xi|^2) |\widehat{\overline{\mathbf{u}}}(\xi, t)|^2 \, d\xi \\ \noalign{\smallskip} & &
   \leq g'(t) \int_{B(t)} |\widehat{\mathbf{u}_0}(\xi)|^2 \,d\xi + g'(t) \:\!\alpha \int_{B(t)} |\xi|^2 |\widehat{\mathbf{u}_0}(\xi)|^2 \,d\xi.
\end{eqnarray*}
Since \(0< P_r(\mathbf{u}_0) < \infty \) and \( r^*(\nabla \mathbf{u}_0) = 1 + r^*(\mathbf{u}_0) \), it follows from the limit definition that there exist \(\rho_0 > 0\) sufficiently small and \(c_2 > 0\) such that for \(0<\rho < \rho_0\),
\[
\rho^{-2r-n} \int_{B(t)} |\widehat{\mathbf{u}_0}(\xi)|^2 \,d\xi \leq c_{2}, \quad \text{ and } \quad \rho^{-2(r+1)-n} \int_{B(t)} |\xi|^2 |\widehat{\mathbf{u}_0}(\xi)|^2 \,d\xi \leq c_2.
\]
Here, \(\rho_0\) and \(c_2\) depend on \( P_r(\mathbf{u}_0) \). Thus:
\[
\frac{d}{dt} \!\left[ g(t) \int_{\mathbb{R}^n} \left( 1 + \alpha |\xi|^2 \right) |\widehat{\overline{\mathbf{u}}}(\xi, t)|^2 \,d\xi \right] \leq c_2 \:\!g'(t) \rho(t)^{2r+n} + \alpha \:\!c_2 \:\!g'(t) \rho(t)^{2(r+1)+n}.
\]
Taking \( g(t) = b^{-k} (t + b)^k \), where \( k > 1 + r + \frac{n}{2} \) and \( b > \frac{k \alpha}{2\mu} \), we have \( g(0) = 1 \), \( g(t) \geq 1 \), \( g'(t) > 0 \), \(2\mu g(t) > \alpha g'(t)\), and \(\rho(t) = (\frac{k}{2 \mu})^{\frac{1}{2}} (t + a)^{-\frac{1}{2}}\), where \(a = b - \frac{k \alpha}{2\mu}\) ($0 < a < b$). Therefore:
\begin{eqnarray*}
    g(t) \cdot (2\pi)^n \|\overline{\mathbf{u}}(t)\|_{H^1_\alpha(\mathbb{R}^n)}^2 &\leq& g(0) \cdot (2\pi)^n \|\mathbf{u}_0\|_{H^1_\alpha(\mathbb{R}^n)}^2 
+ c_2 \int_0^t g'(s) \rho(s)^{2r+n} \, ds \\ &+ & \alpha \:\!c_2 \int_0^t g'(s) \rho(s)^{2(r+1)+n} \, ds \\ &\leq& C + C(t+b)^{k-(r+\frac{n}{2})} + C(t+b)^{k-(r+1+\frac{n}{2})},
\end{eqnarray*}
where \( C = C(\|\mathbf{u}_0\|_{H^1_\alpha(\mathbb{R}^n)}, \alpha, c_2, \mu, a, b, k, r, n) > 0 \), and hence
\begin{eqnarray*}
 \|\overline{\mathbf{u}}(t)\|_{H^1_\alpha(\mathbb{R}^n)}^2 &\leq& C(t+b)^{-k} + C(t+b)^{-(r+\frac{n}{2})} + C(t+b)^{-(r+1+\frac{n}{2})}   \\ \noalign{\smallskip} &\leq& C(t+b)^{-(r+\frac{n}{2})}, 
\end{eqnarray*}
for all $t \geq 0$. So,
\[
\bigl\|\overline{\mathbf{u}}(t)\bigr\|_{H^1_\alpha(\mathbb{R}^n)}^2 \leq C_2(t+C_3)^{-\left(\frac{n}{2} + r^{*}\right)}, \quad \,\,\forall \; t \geq 0.
\]
Therefore, we conclude that if \( -\frac{n}{2} < r^* < \infty \), then there exist constants $C_1$, $C_2$, $C_3 > 0$ such that
\[
C_1(t+1)^{-\left(\frac{n}{2} + r^{*}\right)} \leq \bigl\|\overline{\mathbf{u}}(t)\bigr\|_{H^1_\alpha(\mathbb{R}^n)}^2 \leq C_2(t+C_3)^{-\left(\frac{n}{2} + r^{*}\right)}, \quad \,\,\forall \; t \geq 0,
\]
where \(C_1 = C_1(P_r(\mathbf{u}_0), \alpha, \mu, r^{*}, n) \), \( C_2 = C_2(P_r(\mathbf{u}_0), \alpha, \mu, r^*, n, \|\mathbf{u}_0\|_{H^1_\alpha(\mathbb{R}^n)}) \), and \(C_3 = C_3(\alpha, \mu, r^*, n) \). This proves the theorem. 
\end{proof}

\section{Decay characterization of solutions}
\label{proof-Main1} We start with the proof of a particularly useful lemma.

\begin{lemma}\label{Lema1}
Let \(\mathbf{u}\) be a weak solution to \eqref{eqn:1}. Then for all \(t \geq 0\) and \(\xi \in \mathbb{R}^3\) we have
\begin{equation*}
\left| \mathcal{F} \left\{ \operatorname{rot} (\mathbf{u} - \alpha \Delta \mathbf{u}) \times \mathbf{u} \right\} \!(\xi, t) \right| 
\leq C \!\left(|\xi| + |\xi|^{2}\right)\! \bigl\|\mathbf{u}(t)\bigr\|^{2}_{H^1_\alpha (\mathbb{R}^3)}, 
\end{equation*}
where the constant $C > 0$ depends only on $\alpha$. In particular,
\[
\left| \mathcal{F} \left\{ \operatorname{rot} (\mathbf{u} - \alpha \Delta \mathbf{u}) \times \mathbf{u} \right\} \!(\xi, t) \right| \leq C (|\xi| + |\xi|^2),
\]
where \(C > 0\) depends only on \(\alpha\) and $\|\mathbf{u}_0\|_{H^1_\alpha(\mathbb{R}^3)}$.
\end{lemma}

\begin{proof}
    Multiplying \eqref{eqn:1} by $\mathbf{u}$ in $L^2$ and using the incompressibility condition, we obtain the energy identity
\begin{equation}\label{id_energy}
\frac{d}{dt} \|\mathbf{u}\|_{H^1_\alpha (\mathbb{R}^3)}^2 = -2 \mu \|\nabla \mathbf{u}\|_{L^2 (\mathbb{R}^3)}^2.
\end{equation}
As a consequence,
\begin{equation}\label{energy}
    \|\mathbf{u}(t)\|_{H^1_\alpha (\mathbb{R}^3)} \leq \|\mathbf{u}_0\|_{H^1_\alpha (\mathbb{R}^3)} \leq \|\mathbf{u}_0\|_{H^3 (\mathbb{R}^3)}, \quad \,\,\forall \; t \geq 0.
\end{equation}
Note that, from \eqref{energy}, we have \(\mathbf{u}(t) \in L^2(\mathbb{R}^3)\) for all \(t \geq 0\), hence \(\widehat{\mathbf{u}}(t)\) is well-defined. Using standard vector identities we obtain 
\[
\operatorname{rot}(\mathbf{u} - \alpha \Delta \mathbf{u}) \times \mathbf{u} = (\mathbf{u} \cdot \nabla) \mathbf{u} - \nabla \!\left( \frac{|\mathbf{u}|^2}{2} \right) - \alpha \operatorname{rot}(\Delta \mathbf{u}) \times \mathbf{u}.
\]
Using that $\operatorname{div} \mathbf{u} = 0$ and integrating by parts, we find
\begin{align}
  \left| \mathcal{F} \left\{ (\mathbf{u} \cdot \nabla) \mathbf{u} \right\} (\xi, t) \right| & = \left| \mathcal{F} \left\{ \operatorname{div} (\mathbf{u} \otimes \mathbf{u}) \right\} (\xi, t) \right| = \left| i \xi \cdot \mathcal{F} \{\mathbf{u} \otimes \mathbf{u}\} \right| \nonumber \\  &\leq \|\mathbf{u} \otimes \mathbf{u}\|_{L^1} |\xi| \leq \|\mathbf{u}\|_{L^2}^2 |\xi|, \label{enl1}
\end{align}
and
\begin{equation}
\left| \mathcal{F} \left\{ \nabla \!\left( \frac{|\mathbf{u}|^2}{2} \right) \right\} (\xi, t) \right| = \left| i \mathcal{F} \left\{ \frac{|\mathbf{u}|^2}{2} \right\} \xi \right|  \leq  \frac{1}{2} \||\mathbf{u}|^2\|_{L^1} |\xi| = \frac{1}{2} \|\mathbf{u}\|_{L^2}^2 |\xi|. \label{enl2}
\end{equation}
Moreover, we have that
\begin{align}
   [\operatorname{rot} (\Delta \mathbf{u}) \times \mathbf{u}]_j & = (\mathbf{u} \cdot \nabla) \Delta \mathbf{u}_j - \mathbf{u} \cdot D_j \Delta \mathbf{u}  \nonumber \\ & = \sum_{k, l = 1}^3 (D_l^2 D_k u_j) u_k - \sum_{k, l = 1}^3 (D_l^2 D_j u_k) u_{k}.
\end{align}
Therefore
\begin{eqnarray*}
\mathcal{F} \left\{ [\operatorname{rot} (\Delta \mathbf{u}) \times \mathbf{u}]_j \right\} (\xi, t) &=& \sum_{k, l = 1}^3 \int_{\mathbb{R}^3} e^{-i \xi \cdot x} (D_l^2 D_j u_k) u_k \,dx \\ &  - & \,\sum_{k, l = 1}^3 \int_{\mathbb{R}^3} e^{-i \xi \cdot x} (D_l^2 D_j u_k) u_k \,dx =: \mathcal{I}_1 + \mathcal{I}_{2}.
\end{eqnarray*}
It is easy to deduce through integration by parts that

\begin{align*}
\mathcal{I}_1 & = i \sum_{k,l=1}^3 \xi_k \int_{\mathbb{R}^3} e^{-i \xi \cdot x} u_k D_l^2 u_j \, dx =  - \sum_{k,l=1}^3 \xi_k \xi_l \int_{\mathbb{R}^3} e^{-i \xi \cdot x} u_k D_l u_j \, dx \\&  - i \sum_{k,l=1}^3 \xi_k \int_{\mathbb{R}^3} e^{-i \xi \cdot x}  D_l u_k D_l u_j \, dx,
\end{align*}
and
\begin{align*}
\mathcal{I}_2 & = -i \sum_{k,l=1}^3 \xi_l \int_{\mathbb{R}^3} e^{-i \xi \cdot x} u_k D_l D_j u_k \, dx + \sum_{k,l=1}^3 \int_{\mathbb{R}^3} e^{-i \xi \cdot x} D_l u_k D_l D_j u_k \, dx \\ & = -i \sum_{k,l=1}^3  \xi_l \int_{\mathbb{R}^3} e^{-i \xi \cdot x} u_k D_l D_j u_k \, dx + \frac{1}{2} \sum_{k,l=1}^3 \int_{\mathbb{R}^3} e^{-i \xi \cdot x} D_j[(D_l u_k)^2] \, dx \\ &  = \sum_{k,l=1}^3 \xi_l \xi_j \int_{\mathbb{R}^3} e^{-i \xi \cdot x} u_k D_l u_k \, dx + i \sum_{k,l=1}^3 \xi_l \int_{\mathbb{R}^3} e^{-i \xi \cdot x} D_j u_k D_l u_k \, dx \} \\ & + \frac{i}{2} \sum_{k,l=1}^3 \xi_j \int_{\mathbb{R}^3} e^{-i \xi \cdot x} (D_l u_k)^2 \, dx.
\end{align*}
Then we have
\begin{align*}
 |  \mathcal{F}  & ( (\operatorname{rot} (\Delta \mathbf{u})  \times \mathbf{u} ) _j  ) |  (\xi, t)  \leq \left| \mathcal{I}_{1} \right| + \left| \mathcal{I}_{2} \right| \\ & \leq   \sum_{k, l = 1}^3 |\xi_k| |\xi_l| \int_{\mathbb{R}^3} |u_k| |D_l u_j| \, dx + \sum_{k, l = 1}^3 |\xi_k| \int_{\mathbb{R}^3} |D_l u_k| |D_l u_j| \, dx \\ & + \, \sum_{k, l = 1}^3 |\xi_l| |\xi_j| \int_{\mathbb{R}^3} |u_k| |D_l u_k| \, dx + \sum_{k, l = 1}^3 |\xi_l| \int_{\mathbb{R}^3} |D_j u_k| |D_l u_k| \, dx \\ & + \,  \frac{1}{2} \sum_{k, l = 1}^3 |\xi_j| \int_{\mathbb{R}^3} |D_l u_k|^2 \, dx,
\end{align*}
and thus
\begin{eqnarray}
    \left| \mathcal{F} \left\{ \alpha \operatorname{rot} (\Delta \mathbf{u}) \times \mathbf{u} \right\} \!(\xi, t) \right| &\leq& 2 \alpha |\xi|^2 \|\mathbf{u}\|_{L^2} \|\nabla \mathbf{u}\|_{L^2} + \frac{5 \alpha}{2} |\xi| \|\nabla \mathbf{u}\|_{L^2}^2 \nonumber \\ \noalign{\smallskip} 
&\leq& \alpha |\xi|^2 \|\mathbf{u}\|_{L^2}^2 + \alpha \!\left( \frac{5}{2} |\xi| + |\xi|^2 \right)\!\|\nabla \mathbf{u}(t)\|_{L^2}^{2} . \label{enl3}
\end{eqnarray}
In view of \eqref{enl1}, \eqref{enl2}, and \eqref{enl3}, we obtain the desired result.
\end{proof}

We have now the following estimate in frequency space. 

\begin{proposition}\label{Prop1}
Let \(\mathbf{u}\) be a weak solution to \eqref{eqn:1}. Then
\[
|\widehat{\mathbf{u}}(\xi, t)|^2 \leq 2 e^{2t M_{\alpha, \mu}(\xi)} |\widehat{\mathbf{u}_0}(\xi)|^2 + C(|\xi|^2 + |\xi|^4) \left( \int_0^t \|\mathbf{u}(s)\|_{H^1_\alpha(\mathbb{R}^3)}^2 \,ds \right)^2\!,
\]
for all \(t \geq 0\) and \(\xi \in \mathbb{R}^3\), where $M_{\alpha, \mu}(\xi)$ is given by \eqref{eqn:deft_M}. Here, \(C > 0\) is a constant that depends only on \(\alpha\).
\end{proposition}
\begin{proof}
Taking the Fourier transform of \eqref{eqn:1}, we obtain
\begin{equation*}
		\begin{cases}
    \widehat{\mathbf{u}}_t + \alpha |\xi|^2 \widehat{\mathbf{u}}_t + \mu |\xi|^2 \widehat{\mathbf{u}} + i \xi \widehat{p} = -\mathscr{F} \!\left\{\operatorname{rot}(\mathbf{u} - \alpha \Delta \mathbf{u}) \times \mathbf{u} \right\}\!(\xi, t) =: G(\xi, t), \\ \noalign{\smallskip}
    i \xi \cdot \widehat{\mathbf{u}}(\xi, t) = 0, \quad \xi \in \mathbb{R}^3, \quad t > 0, \\ \noalign{\smallskip}
    \widehat{\mathbf{u}}(\xi, 0) = \widehat{\mathbf{u}_0}(\xi), \quad \xi \in \mathbb{R}^3.
\end{cases} 
\end{equation*}
From this it follows that
\[
(1 + \alpha |\xi|^2) (\partial_t |\widehat{\mathbf{u}}|^2) + 2 \mu |\xi|^2 |\widehat{\mathbf{u}}|^2  \leq 2 | G(\xi, t)| |\widehat{\mathbf{u}} (\xi, t)|,
\]
and thus
\begin{eqnarray*}
    |\widehat{\mathbf{u}}(\xi, t)| &\leq& e^{t M_{\alpha, \mu}(\xi)} |\widehat{\mathbf{u}_0}(\xi)| + \int_0^t e^{(t - s) M_{\alpha, \mu}(\xi)} \:\!\frac{|G(\xi, s)|}{\sqrt{1 + \alpha |\xi|^{2}\,}} \,ds \\ \noalign{\smallskip}
    &\leq& e^{t M_{\alpha, \mu}(\xi)} |\widehat{\mathbf{u}_0}(\xi)| + \int_0^t e^{(t - s) M_{\alpha, \mu}(\xi)} \:\!|G(\xi, s)| \,ds,
\end{eqnarray*}
where $M_{\alpha, \mu}$ is given by \eqref{eqn:deft_M}. By Lemma \ref{Lema1}, we have
\[
\bigl|G(\xi, t)\bigr| \leq C \!\left(|\xi| + |\xi|^{2}\right)\! \bigl\|\mathbf{u}(t)\bigr\|^{2}_{H^1_\alpha (\mathbb{R}^3)},
\]
for all \(t \geq 0\) and \(\xi \in \mathbb{R}^3\), where \(C > 0\) depends only on \(\alpha\). This completes the proof of Proposition \ref{Prop1}.
\end{proof}

We present now the proof of the first main result of this paper. 


\begin{proof}[\bf Proof of Theorem \ref{Main1-Intro}] We will again use the Fourier splitting method. We define the closed ball centered at the origin with radius \(\rho(t)\)
\[
B(t) := \left\{ \xi \in \mathbb{R}^3 \,; \, |\xi| \leq \rho(t) \right\}\!,
\]
where
\[
\rho(t) := \left( \frac{g'(t)}{2 \mu g(t) - \alpha g'(t)} \right)^{\frac{1}{2}}\!,
\]
and \(g\) is a differentiable function on \((0, \infty)\) satisfying
\[
g(0) = 1, \quad g(t) \geq 1, \quad g'(t) > 0, \quad \text{and} \quad 2 \mu g(t) > \alpha g'(t), \quad \,\,\forall \; t  > 0.
\]
Using Plancherel's Theorem, it follows from the energy identity \eqref{id_energy} that
\[
\frac{d}{dt} \!\left( \|\widehat{\mathbf{u}}\|_{L^2}^2 + \alpha \|\widehat{\nabla \mathbf{u}}\|_{L^2}^2 \right) = -2 \mu \|\widehat{\nabla \mathbf{u}}\|_{L^2}^2,
\]
i.e.,
\[
\frac{d}{dt} \!\left( \int_{\mathbb{R}^3} (1 + \alpha |\xi|^2) |\widehat{\mathbf{u}}|^2 \, d\xi \right) = -2 \mu \int_{\mathbb{R}^3} |\xi|^2 |\widehat{\mathbf{u}}|^2 \, d\xi.
\]
Multiplying the above identity by \(g(t)\), we obtain
\begin{eqnarray*}
    \frac{d}{dt} \!\left( g(t) \int_{\mathbb{R}^3} (1 + \alpha |\xi|^2) |\widehat{\mathbf{u}}|^2 \, d\xi \right) &=& g'(t) \int_{\mathbb{R}^3} (1 + \alpha |\xi|^2) |\widehat{\mathbf{u}}|^2 \, d\xi \\ \noalign{\medskip} &  - & \, 2 \mu g(t) \int_{\mathbb{R}^3} \frac{|\xi|^2 }{1 + \alpha |\xi|^2} (1 + \alpha |\xi|^2) |\widehat{\mathbf{u}}|^2 \, d\xi.
\end{eqnarray*}
From the definition of \(B(t)\), we have
\[
\frac{2 \mu g(t) |\xi|^2}{1 + \alpha |\xi|^2} > g'(t), \quad \,\,\forall \; \xi \in B(t)^c \equiv \mathbb{R}^{3} \, \backslash \, B(t).
\]
Thus,
\begin{align*}
    \frac{d}{dt} & \!\left( g(t) \int_{\mathbb{R}^3} (1 + \alpha |\xi|^2) |\widehat{\mathbf{u}}(\xi, t)|^2 \, d\xi \right)  \leq g'(t) \int_{\mathbb{R}^3} (1 + \alpha |\xi|^2) |\widehat{\mathbf{u}}(\xi, t)|^2 \, d\xi   \\ & -  2 \mu g(t) \int_{B(t)^c} \frac{|\xi|^2}{1 + \alpha |\xi|^2} (1 + \alpha |\xi|^2) |\widehat{\mathbf{u}}(\xi, t)|^2 \, d\xi  \\ & \leq g'(t) \int_{B(t)} (1 + \alpha |\xi|^2) |\widehat{\mathbf{u}}(\xi, t)|^2 \, d\xi,
\end{align*}
that is,
\[
\frac{d}{dt} \!\left( g(t) \int_{\mathbb{R}^3} (1 + \alpha |\xi|^2) |\widehat{\mathbf{u}}(\xi, t)|^2 \, d\xi \right) \leq g'(t) \int_{B(t)} (1 + \alpha |\xi|^2) |\widehat{\mathbf{u}}(\xi, t)|^2 \, d\xi.
\]
Integrating the above inequality with respect to time, we obtain
\begin{equation}\label{星号}
    g(t) \int_{\mathbb{R}^3} (1 + \alpha |\xi|^2) |\widehat{\mathbf{u}}(\xi, t)|^2 \, d\xi \leq I_1 + I_2(t),
\end{equation}
where
\[
I_1 := g(0) \int_{\mathbb{R}^3} (1 + \alpha |\xi|^2) |\widehat{\mathbf{u}_0}(\xi)|^2 \, d\xi,
\]
and
\[
I_2(t) := \int_0^t \left( g'(s) \int_{B(s)} (1 + \alpha |\xi|^2) |\widehat{\mathbf{u}}(\xi, s)|^2 \, d\xi \right) \! ds.
\]
Note that, by Plancherel's theorem, we have
\begin{eqnarray*}
    & & \int_{\mathbb{R}^3} (1 + \alpha |\xi|^2) |\widehat{\mathbf{u}_0}(\xi)|^2 \, d\xi = \int_{\mathbb{R}^3} |\widehat{\mathbf{u}_0}(\xi)|^2 \, d\xi + \alpha \int_{\mathbb{R}^3} |\xi|^2 |\widehat{\mathbf{u}_0}(\xi)|^2 \, d\xi \\ \noalign{\medskip} & & = \int_{\mathbb{R}^3} |\widehat{\mathbf{u}}_0(\xi)|^2 \, d\xi + \alpha \int_{\mathbb{R}^3} |\widehat{\nabla \mathbf{u}}_0(\xi)|^2 \, d\xi = \|\widehat{\mathbf{u}_0}\|_{L^2}^2 + \alpha \|\widehat{\nabla \mathbf{u}_0}\|_{L^2}^2 \\ \noalign{\medskip} & & = (2 \pi)^{3} \!\left(\|\mathbf{u}_0\|_{L^2}^2 + \alpha \|\nabla \mathbf{u}_0\|_{L^2}^{2}\right) = (2 \pi)^{3} \|\mathbf{u}_0\|_{H^1_\alpha(\mathbb{R}^3)}^{2}.
\end{eqnarray*}
Therefore,
\begin{equation}\label{contaI1}
I_1 = g(0) \cdot (2\pi)^3 \:\!\|\mathbf{u}_0\|_{H^1_\alpha(\mathbb{R}^3)}^2 = (2\pi)^3 \:\!\|\mathbf{u}_0\|_{H^1_\alpha(\mathbb{R}^3)}^2,   
\end{equation}
since \(g(0) = 1\). Next, we will estimate \(I_2(t)\). By the Proposition \ref{Prop1}, we have
\begin{eqnarray*}
 & & I_2(t) \leq 2 \int_0^t g'(s) \!\left(\int_{B(s)} (1 + \alpha |\xi|^2) e^{2s M_{\alpha, \mu}(\xi)} |\widehat{\mathbf{u}_0}(\xi)|^2 \, d\xi\right)  \! ds   \\ \noalign{\medskip} & & + \,C \int_0^t g'(s) \!\left(\int_{B(s)} (1 + \alpha |\xi|^2) (|\xi|^2 + |\xi|^4) \!\left( \int_0^s \|\mathbf{u}(\tau)\|_{H^1_\alpha(\mathbb{R}^3)}^2 \, d\tau \right)^{2}  d\xi \right)\! ds \\ \noalign{\medskip} & & =: I_{2,1}(t) + I_{2,2}(t).
\end{eqnarray*}
Initially, we observe that, by Plancherel
\begin{align*}
    \int_{B(s)} & (1 + \alpha |\xi|^2) e^{2s M_{\alpha,\mu} (\xi)}|\widehat{\mathbf{u}_0}(\xi)|^2 \, d\xi = \int_{B(s)} (1 + \alpha |\xi|^2) |\widehat{\overline{\mathbf{u}}}(\xi, s)|^2 \,d\xi \\  
&\leq \int_{\mathbb{R}^3} (1 + \alpha |\xi|^2) |\widehat{\overline{\mathbf{u}}}(\xi, s)|^2 \,d\xi = \int_{\mathbb{R}^3} |\widehat{\overline{\mathbf{u}}}(\xi, s)|^2 \,d\xi + \alpha \int_{\mathbb{R}^3} |\xi|^2 |\widehat{\overline{\mathbf{u}}}(\xi, s)|^2 \, d\xi \\  & = \int_{\mathbb{R}^3} |\widehat{\overline{\mathbf{u}}}(\xi, s)|^2 \, d\xi + \alpha \int_{\mathbb{R}^3} |\widehat{\nabla \overline{\mathbf{u}}}(\xi, s)|^2 \, d\xi = \|\widehat{\overline{\mathbf{u}}}(s)\|_{L^2}^2 + \alpha \|\widehat{\nabla \:\! \overline{\mathbf{u}}}(s)\|_{L^2}^2  \\ & = (2\pi)^3 \left(\|\overline{\mathbf{u}}(s)\|_{L^2}^2 + \alpha \|\nabla \:\! \overline{\mathbf{u}}(s)\|_{L^2}^2 \right) = (2\pi)^3 \:\! \|\overline{\mathbf{u}}(s) \|_{H^1_{\alpha} (\mathbb{R}^3)}^2.
\end{align*}
Therefore, by Theorem \ref{teodeclupl158}, we have
\begin{equation}\label{contaI21}
I_{2,1}(t) \leq 2  (2\pi)^3 \int_0^t g'(s) \|\overline{\mathbf{u}} (s)\|_{H^1_\alpha (\mathbb{R}^3)}^2 \, ds \leq 16 \pi^3 \:\! C_2 \int_0^t g'(s) \!\left(s + C_{3}\right)^{-\left(\frac{3}{2} + r^*\right)} \, ds.    
\end{equation}
On the other hand, using polar coordinates, it follows that
\begin{equation}\label{contaI22}
I_{2,2}(t) \leq C \!\left( \int_0^t g'(s) (\rho(s)^5 + \rho(s)^7 + \rho(s)^9) ds \right) \! \left( \int_0^t \|\mathbf{u}(\tau)\|_{H^1_\alpha(\mathbb{R}^3)}^2 \,d\tau \right)^2.
\end{equation}
For a fixed \(r^* \in (-\frac{3}{2}, \infty)\), we choose
\[
g(t) = \widetilde{b}^{-m} (t + \widetilde{b})^{m}, \quad \text{where} \quad m > \max \left\{ \frac{5}{2}, \frac{3}{2} + r^* \right\}\!, \quad \text{and} \quad \widetilde{b} \geq 1 + \frac{\alpha \:\! m}{2 \mu}.
\]
This way,
\begin{equation}\label{esc_ro}
   \rho(t) = \left(\frac{m}{2 \mu}\right)^{\frac{1}{2}} \!\left(t + \widetilde{a}\right)^{-\frac{1}{2}}, \quad \text{where} \quad \widetilde{a} := \widetilde{b} - \frac{\alpha \:\! m}{2 \mu} \geq 1. 
\end{equation}
From this, it follows from \eqref{contaI21} and \eqref{contaI22}, respectively, that
\begin{eqnarray}\label{dag}
    I_{2,1}(t) &\leq& 16 \pi^3 \:\!C_2 \widetilde{b}^{-m} \:\!m \int_0^t (s + \widetilde{b})^{m-1} (s + C_3)^{-\left(\frac{3}{2} + r^*\right)} \,ds \nonumber \\ \noalign{\smallskip} 
    &\leq& C_4 \int_0^t (s + C_5)^{m-1 - \left(\frac{3}{2} + r^*\right)} \,ds \nonumber \\ \noalign{\smallskip} 
    &\leq& C_4 (t + C_5)^{m - \left(\frac{3}{2} + r^*\right)},
\end{eqnarray}
and
\begin{eqnarray}\label{dag_doble}
I_{2,2}(t) &\leq& C_6 \int_0^t (s + \widetilde{b})^{m-1} \!\left( (s + \widetilde{a})^{-\frac{5}{2}} + (s + \widetilde{a})^{-\frac{7}{2}} + (s + \widetilde{a})^{-\frac{9}{2}} \right) \!ds \nonumber \\ \noalign{\smallskip} & & \times \left( \int_0^t \|\mathbf{u}(\tau)\|_{H^1_\alpha(\mathbb{R}^3)}^2 \,d\tau \right)^2 \nonumber \\ \noalign{\smallskip} &\leq& C_7 \left( \int_0^t (s + C_8)^{m-1 - \frac{5}{2}} ds \right) \left( \int_0^t \|\mathbf{u}(\tau)\|_{H^1_\alpha(\mathbb{R}^3)}^2 \,d\tau \right)^2 \nonumber \\ \noalign{\smallskip} &\leq& C_7 (t+ C_{8})^{m - \frac{5}{2}} \left( \int_0^t \|\mathbf{u}(\tau)\|_{H^1_\alpha(\mathbb{R}^3)}^2 \,d\tau \right)^{2}\!.
\end{eqnarray}
Hence, since
\[
\int_{\mathbb{R}^3} (1 + \alpha |\xi|^2) |\widehat{\mathbf{u}}(\xi, t)|^2 d\xi = (2\pi)^3 \|\mathbf{u}(t)\|_{H^1_\alpha(\mathbb{R}^3)}^2,
\]
it follows from \eqref{星号}, \eqref{contaI1}, \eqref{dag} and \eqref{dag_doble} that
\begin{align}
\label{§}
 (2\pi)^3 g(t) \|\mathbf{u}(t)\|_{H^1_\alpha(\mathbb{R}^3)}^2 & \leq (2\pi)^3 \|\mathbf{u}_0\|_{H^1_\alpha(\mathbb{R}^3)}^2  +  C_4 (t + C_5)^{m - (\frac{3}{2} + r^*)} \notag \\ &  + C_7 (t + C_8)^{m - \frac{5}{2}} \left( \int_0^t \|\mathbf{u}(\tau)\|_{H^1_\alpha(\mathbb{R}^3)}^2 \,d\tau \right)^2.
\end{align}
Initially, using \eqref{energy}, we obtain
\begin{displaymath}
g(t) \|\mathbf{u}(t)\|_{H^1_\alpha(\mathbb{R}^3)}^2 \leq \|\mathbf{u}_0\|_{H^1_\alpha(\mathbb{R}^3)}^2 + C_4 (t + C_5)^{m - (\frac{3}{2} + r^*)}  +\, C_7 (t + C_8)^{m - \frac{5}{2}+ 2},    
\end{displaymath}
and consequently
\begin{eqnarray}\label{corazon}
\|\mathbf{u}(t)\|_{H^1_\alpha(\mathbb{R}^3)}^2 &\leq& C (t + \widetilde{C})^{-m} + C (t + \widetilde{C})^{-(\frac{3}{2} + r^*)} + C (t + \widetilde{C})^{- \frac{1}{2}} \nonumber \\ \noalign{\smallskip} &\leq& \widetilde{C}_1 (t + \widetilde{C}_2)^{-\min\left\{\frac{3}{2} + r^*, \frac{1}{2}\right\}}, \quad \,\,\forall \; t \geq 0,
\end{eqnarray}
since $g(t) = \widetilde{b}^{-m} (t + \widetilde{b})^{m}$. This estimate preliminary can be further improved, as we'll show below. Suppose that \(\min \!\left\{\frac{3}{2} + r^*, \frac{1}{2}\right\} = \frac{3}{2} + r^*\), i.e., \(-\frac{3}{2} < r^* \leq - 1\). In this case, we have from \eqref{corazon} that
\begin{eqnarray*}
 \left( \int_0^t \|\mathbf{u}(s)\|_{H^1_\alpha(\mathbb{R}^3)}^2 \,ds \right)^2 &\leq& \widetilde{C}_1^2 \left( \int_0^t (s + \widetilde{C}_2)^{-(\frac{3}{2} + r^*)} \,ds \right)^2    \\ \noalign{\smallskip} &\leq& \widetilde{C}_1^2 \left( (t + \widetilde{C}_2)^{-(\frac{1}{2} + r^*)} \right)^2 = \widetilde{C}_1^2 (t + \widetilde{C}_2)^{-(1 + 2r^*)}.
\end{eqnarray*}
Therefore, remembering the choice of $g$, we have by \eqref{§}, that
\begin{eqnarray*}
\|\mathbf{u}(t)\|_{H^1_\alpha(\mathbb{R}^3)}^2 &\leq& C (t + \widetilde{C})^{-m} + C (t + \widetilde{C})^{-(\frac{3}{2} + r^*)} + C (t + \widetilde{C})^{-(\frac{5}{2} + 1 + 2r^*)} \\ \noalign{\smallskip} &\leq& C (t + \widetilde{C})^{-(\frac{3}{2} + r^*)} + C (t + \widetilde{C})^{-(\frac{7}{2} + 2r^*)} \\ \noalign{\smallskip} &\leq& C (t + \widetilde{C})^{-(\frac{3}{2} + r^*)}. 
\end{eqnarray*}
On the other hand, if \(\min \!\left\{\frac{3}{2} + r^*, \frac{1}{2}\right\} = \frac{1}{2}\), i.e., \(r^* \geq -1\), it follows by \eqref{corazon} that

\[
\left( \int_0^t \|\mathbf{u}(s)\|_{H^1_\alpha(\mathbb{R}^3)}^2 \,ds \right)^2 \leq \widetilde{C}_1^2 \left( \int_0^t (s + \widetilde{C}_2)^{-\frac{1}{2}} \,ds \right)^2 \leq \widetilde{C}_1^2 (t + \widetilde{C}_2).
\]
Again from \eqref{§}, we obtain
\begin{eqnarray*}
    \|\mathbf{u}(t)\|_{H^1_\alpha(\mathbb{R}^3)}^2 &\leq& C (t + \widetilde{C})^{-m} + C (t + \widetilde{C})^{-\left(\frac{3}{2} + r^*\right)} + C (t + \widetilde{C})^{-\frac{5}{2} + 1} \\ \noalign{\smallskip} &\leq& C (t + \widetilde{C})^{-\left(\frac{3}{2} + r^*\right)}  + C (t + \widetilde{C})^{-\frac{3}{2}}  \\ \noalign{\smallskip} &\leq& \left\{\begin{array}{l}
C(t+\widetilde{C})^{-\left(\frac{3}{2}+r^*\right),}, \text { if } -1 \leq r^* \leq 0 ,\\ \noalign{\smallskip}
C(t+\widetilde{C})^{-\frac{3}{2}}, \text { if } r^* \geq 0.
\end{array}\right.
\end{eqnarray*}
Therefore, from the above considerations, we have
\begin{equation}\label{oro}
\|\mathbf{u}(t)\|_{H^1_\alpha(\mathbb{R}^3)}^2 \leq K_1 (t + K_2)^{-\min\left\{\frac{3}{2} + r^*, \frac{3}{2}\right\}}, \quad \,\,\forall \; t \geq 0. \end{equation}
It is important to note that $K_2 \geq 1$ (see \eqref{esc_ro}). We will apply the ``bootstrap'' argument once more. Suppose that \(\min \!\left\{\frac{3}{2} + r^*, \frac{3}{2}\right\} = \frac{3}{2} + r^*\), i.e., \(r^* \leq 0\). 
With \eqref{oro} in mind, we will divide the analysis into three cases: \newline\newline 
\noindent (i) If \(-\frac{3}{2} < r^* < -\frac{1}{2}\), we then have
\[
\left( \int_0^t \|\mathbf{u}(s)\|_{H^1_\alpha(\mathbb{R}^3)}^2 \,ds \right)^2 \leq K_1^2 (t + K_2)^{-(1 + 2r^*)}.
\]

\noindent (ii) If \(r^* = -\frac{1}{2}\), we then have (since $K_2 \geq 1$)
\[
\left( \int_0^t \|\mathbf{u}(s)\|_{H^1_\alpha(\mathbb{R}^3)}^2 \,ds \right)^2 \leq K_1^2 \ln^2(t + K_2).
\]

\noindent (iii) If \(-\frac{1}{2} < r^* \leq 0\), we then have
\[
\left( \int_0^t \|\mathbf{u}(s)\|_{H^1_\alpha(\mathbb{R}^3)}^2 \,ds \right)^2 \leq C.
\]
So, by \eqref{§} and by the choice of $g$:
\newline\newline 
\noindent (i) If \(
- \frac{3}{2} < r^* < -\frac{1}{2}
\), then
\begin{eqnarray*}
    \|\mathbf{u}(t)\|_{H^1_\alpha(\mathbb{R}^3)}^2 &\leq& C (t + \widetilde{C})^{-m} + C (t + \widetilde{C})^{-(\frac{3}{2} + r^*)} + C (t + \widetilde{C})^{-\frac{5}{2} - 1 - 2r^*} \\ \noalign{\smallskip} &\leq& C (t + \widetilde{C})^{-(\frac{3}{2} + r^*)} + C (t + \widetilde{C})^{-(\frac{7}{2} + 2r^*)} \\ \noalign{\smallskip} &\leq& C (t + \widetilde{C})^{-(\frac{3}{2} + r^*)}.
\end{eqnarray*}

\noindent (ii) If \(r^* = -\frac{1}{2}\), then
\begin{eqnarray*}
\|\mathbf{u}(t)\|_{H^1_\alpha(\mathbb{R}^3)}^2 &\leq& C (t + \widetilde{C})^{-m} + C (t + \widetilde{C})^{-(\frac{3}{2} + r^*)} + C (t + \widetilde{C})^{-\frac{5}{2}} \ln^2(t + \widetilde{C}) \\ \noalign{\smallskip} &\leq& C (t + \widetilde{C})^{-1}  + C (t + \widetilde{C})^{-\frac{5}{2}} \ln^2(t + \widetilde{C}) \\ \noalign{\smallskip} &\leq& C (t + \tilde{C})^{-1} =  C (t + \widetilde{C})^{-(\frac{3}{2} + r^*)}. 
\end{eqnarray*}

\noindent (iii) If \(-\frac{1}{2} < r^* \leq 0\), then
\begin{eqnarray*}
 \|\mathbf{u}(t)\|_{H^1_\alpha(\mathbb{R}^3)}^2 &\leq& C (t + \widetilde{C})^{-m} + C (t + \widetilde{C})^{-(\frac{3}{2} + r^*)} + C (t + \widetilde{C})^{-\frac{5}{2}}   \\ \noalign{\smallskip} &\leq& C (t + \widetilde{C})^{-(\frac{3}{2} + r^*)}.
\end{eqnarray*}
On the other hand, if \(\min\!\left\{\frac{3}{2} + r^*, \frac{3}{2}\right\} = \frac{3}{2}\), i.e., \(r^* \geq 0\), then, by \eqref{oro}:
\[
\left( \int_0^t \|\mathbf{u}(s)\|_{H^1_\alpha(\mathbb{R}^3)}^2 \,ds \right)^2 \leq C,
\]
and, again by \eqref{§}, we find
\begin{eqnarray*}
\|\mathbf{u}(t)\|_{H^1_\alpha(\mathbb{R}^3)}^2 &\leq& C (t + \widetilde{C})^{-m} + C (t + \widetilde{C})^{-(\frac{3}{2} + r^*)} + C (t + \widetilde{C})^{-\frac{5}{2}} \\ \noalign{\smallskip}  
&\leq& C (t + \widetilde{C})^{-(\frac{3}{2} + r^*)} + C (t + \widetilde{C})^{-\frac{5}{2}} \\ \noalign{\smallskip} 
&\leq& \left\{\begin{array}{l}
C(t+\widetilde{C})^{-\left(\frac{3}{2}+r^*\right),}, \text { if } 0 \leq r^* \leq 1 ,\\ \noalign{\smallskip}
C(t+\widetilde{C})^{-\frac{5}{2}}, \text { if } r^* \geq 1.
\end{array}\right.
\end{eqnarray*}
This ends the proof of the theorem.
\end{proof}

\section{Asymptotic equivalence}

\label{proof-Main2}

The estimate from  Theorem \ref{Main1-Intro} enables us to determine the decay rate for the difference between a solution $\mathbf{u}$ of \eqref{eqn:1} and the solution $\overline{\mathbf{u}}$ to the linear part \eqref{plppe} with the same initial data. However, due to the presence of the regularizing term $- \alpha \:\!\partial_{t} \Delta \mathbf{u}$ in \eqref{eqn:1}, we need more regularity for the initial data. 
%

Before proving Theorem \ref{Main2-intro}, we need the following auxiliary result for the derivatives of the solutions of the linear problem.

\begin{proposition}[Decay of the derivatives of linear part]\label{teoder128}
Let $\overline{\mathbf{u}}$ be the solution to the linear part \eqref{plppe}, with \( D^{m} \:\! \mathbf{u}_0 \in H^{1}_\alpha(\mathbb{R}^3) \), $m \in \mathbb{N} \cup \{0\}$, and \(\operatorname{div} \mathbf{u}_0 = 0\). Then,
$$\bigl\|D^m \:\!\overline{\mathbf{u}}(t)\bigr\|^2_{H^1_\alpha(\mathbb{R}^3)} \leq \widetilde{C}_{2}(t+\widetilde{C}_{3})^{-\left(\frac{3}{2}+r^{*} +m\right)}, \quad \,\,\forall \; t \geq 0.$$ 
\end{proposition}

\begin{proof}
Applying the derivation operator $D^m$ to \eqref{plppe}, we obtain
\[
\partial_t (D^m \:\!\overline{\mathbf{u}} - \alpha \Delta D^m \:\!\overline{\mathbf{u}}) - \mu \Delta D^m \:\!\overline{\mathbf{u}} = 0.
\]
Here, again we will use the Fourier splitting method. Multiplying (dot product) the equation above by \( D^m \:\!\overline{\mathbf{u}} \) and integrating in \( \mathbb{R}^3 \), we find
\[
\frac{1}{2} \frac{d}{dt} \!\left(\| D^m \:\!\overline{\mathbf{u}} \|_{L^2}^2 + \alpha \| D^{m+1} \:\!\overline{\mathbf{u}} \|_{L^2}^2 \right) + \mu \| D^{m+1} \:\!\overline{\mathbf{u}} \|_{L^2}^2 = 0.
\]
Using Plancherel, we get
\begin{eqnarray*}
  \frac{d}{dt} \| \widehat{D^m \:\!\overline{\mathbf{u}}}(t) \|_{H^1_\alpha(\mathbb{R}^3)}^2 &=& -2\mu \| \widehat{D^{m+1} \:\!\overline{\mathbf{u}}}(t) \|_{L^2}^2   \\ \noalign{\smallskip} &=& -2\mu \int_{\mathbb{R}^3} |\xi|^{2m+2} \:\!|\widehat{\overline{\mathbf{u}}}(\xi, t)|^2 \,d\xi.
\end{eqnarray*}
Multiplying the above inequality by \( g(t) \), we obtain
\begin{eqnarray*}
&\displaystyle \frac{d}{dt} \!\left[ g(t) \| \widehat{D^m \:\!\overline{\mathbf{u}}}(t) \|_{H^1_\alpha(\mathbb{R}^3)}^2 \right] = g'(t) \| \widehat{D^m \:\!\overline{\mathbf{u}}}(t) \|_{H^1_\alpha(\mathbb{R}^3)}^2
- 2\mu g(t) \int_{\mathbb{R}^3} |\xi|^{2m+2} |\widehat{\overline{\mathbf{u}}}(\xi, t)|^2 \,d\xi & \\ \noalign{\smallskip} &\displaystyle \leq g'(t) \| \widehat{D^m \:\!\overline{\mathbf{u}}}(t) \|_{H^1_\alpha(\mathbb{R}^3)}^2
- 2\mu g(t)\int_{B(t)^c} |\xi|^{2m+2} |\widehat{\overline{\mathbf{u}}}(\xi, t)|^2 \,d\xi,&
\end{eqnarray*}
and consequently
\begin{eqnarray*}
\frac{d}{dt} \!\left[ g(t) \| \widehat{D^m \:\!\overline{\mathbf{u}}}(t) \|_{H^1_\alpha(\mathbb{R}^3)}^2 \right] &\leq& 
g'(t) \int_{\mathbb{R}^3} (1 + \alpha |\xi|^2) |\xi|^{2m} |\widehat{\overline{\mathbf{u}}}(\xi, t)|^2 \,d\xi \\ \noalign{\medskip} & - &
 \,2\mu g(t) \int_{B(t)^c} \frac{|\xi|^{2m+2}}{1+\alpha |\xi|^2} (1+\alpha |\xi|^2) |\widehat{\overline{\mathbf{u}}}(\xi, t)|^2 \,d\xi.
\end{eqnarray*}
For \( \xi \in B(t)^c \), we have
\[
\frac{2\mu g(t) |\xi|^2}{1+\alpha |\xi|^2} > g'(t).
\]
Thus, considering \eqref{8924lp}, we have
\begin{eqnarray*}
\frac{d}{dt} \!\left[ g(t) \| \widehat{D^m \:\!\overline{\mathbf{u}}}(t) \|_{H^1_\alpha(\mathbb{R}^3)}^2 \right] &\leq& 
g'(t) \int_{B(t)} (1+\alpha |\xi|^2) |\xi|^{2m} |\widehat{\overline{\mathbf{u}}}(\xi, t)|^2 \,d\xi \\ \noalign{\smallskip} &\leq& g'(t) \int_{B(t)} |\xi|^{2m} |\widehat{\mathbf{u}}_0(\xi)|^2 \,d\xi \\ \noalign{\smallskip} & + &  \, \alpha \:\!g'(t) \int_{B(t)} |\xi|^{2(m+1)} |\widehat{\mathbf{u}}_0(\xi)|^2 \,d\xi.
\end{eqnarray*}
Since
\[
0<P_r(\mathbf{u}_0) < \infty, \quad r^*(D^m \mathbf{u}_0) = m + r^*(\mathbf{u}_0), \quad \text{and} \quad r^*(D^{m+1} \mathbf{u}_0) = m+1 + r^*(\mathbf{u}_0),
\]
it follows from the definition of the limit that there exist \(\widetilde{\rho}_0 > 0\) sufficiently small and \(\widetilde{c} > 0\) such that for \(0 < \rho < \widetilde{\rho}_0\):
\[
\rho^{-2(r+m)-3} \int_{B(t)} |\xi|^{2m} |\widehat{\mathbf{u}}_0(\xi)|^2 \, d\xi \leq \widetilde{c},
\]
and
\[
\rho^{-2(r+m+1)-3} \int_{B(t)} |\xi|^{2(m+1)} |\widehat{\mathbf{u}}_0(\xi)|^2 \, d\xi \leq \widetilde{c}.
\]
Here, \(\widetilde{\rho}_0\) and \(\widetilde{c}\) depend on \(P_r(\mathbf{u}_0)\). Thus:
\[
\frac{d}{dt} \!\left[ g(t) \| \widehat{D^m \:\!\overline{\mathbf{u}}}(t) \|_{H^1_\alpha(\mathbb{R}^3)}^2 \right] \leq \widetilde{c} \:\!g'(t) \rho(t)^{2(r+m)+3} +
\alpha \:\!\widetilde{c} \:\!g'(t) \rho(t)^{2(r+m+1)+3}.
\]
By Plancherel's theorem, we get
\begin{eqnarray*}
    \frac{d}{dt} \!\left[ g(t) \| D^m \:\!\overline{\mathbf{u}}(t) \|_{H^1_\alpha(\mathbb{R}^3)}^2 \right] &\leq& (2\pi)^{-3} \:\!\widetilde{c} \:\!g'(t) \rho(t)^{2(r+m)+3} \\ \noalign{\smallskip} & + & 
(2\pi)^{-3} \:\!\alpha \:\!\widetilde{c} \:\!g'(t) \rho(t)^{2(r+m+1)+3}.
\end{eqnarray*}
Now, taking \( g(t) = b^{-k} (t+b)^k \), where \( k > 1+r+m + \frac{3}{2} \) and \( b > \frac{k \:\! \alpha }{2\mu} \), we have \( g(0) = 1 \), \( g(t) \geq 1 \), \( g'(t) > 0 \), \( 2\mu g(t) > \alpha g'(t) \), and

\[
\rho(t) = \left(\frac{k}{2 \mu}\right)^{\frac{1}{2}} (t+a)^{-\frac{1}{2}},
\]
where \( a = b - \displaystyle\frac{k \:\!\alpha}{2 \mu} \) (\(0 < a < b\)). Therefore, integrating the last differential inequality above over $(0, t)$, we obtain
\begin{eqnarray*}
g(t) \|D^m \:\!\overline{\mathbf{u}}(t)\|^2_{H^1_\alpha(\mathbb{R}^3)} &\leq& g(0)  \|D^m \:\!\mathbf{u}_0 \|^2_{H^1_\alpha(\mathbb{R}^3)} +
(2\pi)^{-3} \:\!\widetilde{c} \int_0^t g'(s) \rho(s)^{2(r+m)+3} \, ds     \\ \noalign{\smallskip} & + & (2\pi)^{-3} \:\!\alpha \:\!\widetilde{c} \int_0^t g'(s) \rho(s)^{2(r+m+1)+3} \, ds.
\end{eqnarray*}
Since \( g(0) = 1 \) and \( D^{m} \:\! \mathbf{u}_0 \in H^{1}_\alpha(\mathbb{R}^3) \), it follows that
\[
g(t) \|D^m \:\!\overline{\mathbf{u}}(t)\|^2_{H^1_\alpha(\mathbb{R}^3)} \leq C + C(t+b)^{k-(r+m+\frac{3}{2})} + C(t+b)^{k-(r+m+1+\frac{3}{2})},
\]
and hence
\begin{eqnarray*}
    \|D^m \:\!\overline{\mathbf{u}}(t)\|^2_{H^1_\alpha(\mathbb{R}^3)} &\leq& C(t+b)^{-k} + C(t+b)^{-(r+m+\frac{3}{2})} + C(t+b)^{-(r+m+\frac{5}{2})} \\ \noalign{\smallskip} &\leq& C(t+b)^{-(r+m+\frac{3}{2})}, \quad \,\,\forall \; t \geq 0.
\end{eqnarray*}
This finishes the proof of the proposition.
\end{proof}

\begin{remark} \label{remark-H4}
As we will see in the proof of the next theorem, in addition to using Theorem \ref{Main1-Intro}, we will use Proposition \ref{teoder128} with $m = 1$ and $m = 3$. In other words, we need $\mathbf{u}_0$, $\nabla \:\! \mathbf{u}_0$, $D^{3} \:\! \mathbf{u}_0 \in H^{1}_\alpha(\mathbb{R}^3)$. This means that we must have $\mathbf{u}_0 \in H^{4}(\mathbb{R}^3)$. That is why we increase the regularity of the initial data in the Theorem \ref{Main2-intro}.   
\end{remark}

Now, we are in position to establish the proof of the asymptotic equivalence result. 

\begin{proof}[\bf Proof of Theorem \ref{Main2-intro}]

Let \(\mathbf{w} := \mathbf{u} - \overline{\mathbf{u}}\), where $\overline{\mathbf{u}}$ is the solution to the linear part \eqref{plppe} with the same initial data $\mathbf{u}_0 \in H^{4}(\mathbb{R}^3)$. Then, in \(\mathbb{R}^3 \times (0, \infty)\), we have
\begin{equation}\label{diffe}
\begin{cases}
\partial_t (\mathbf{w} - \alpha \Delta \mathbf{w}) - \mu \Delta \mathbf{w} + \text{rot} (\mathbf{u} - \alpha \Delta \mathbf{u}) \times \mathbf{u} + \nabla p = {\boldsymbol 0}, \\  \noalign{\smallskip}
    \text{div} \, \mathbf{w} = 0, \\ \noalign{\smallskip}
    \mathbf{w}(0) = {\boldsymbol 0}.
\end{cases} 
\end{equation}
Applying the Fourier transform, we find
\begin{equation*}
		\begin{cases}
(1 + \alpha |\xi|^2) \widehat{\mathbf{w}}_t + \mu |\xi|^2 \widehat{\mathbf{w}} + i \xi \widehat{p} = -\mathscr{F} \!\left\{\operatorname{rot}(\mathbf{u} - \alpha \Delta \mathbf{u}) \times \mathbf{u} \right\}\!(\xi, t) =: G(\xi, t), \\ \noalign{\smallskip}
i \xi \cdot \widehat{\mathbf{w}}(\xi, t) = 0, \quad \xi \in \mathbb{R}^3, \quad t > 0, \\ \noalign{\smallskip}
\widehat{\mathbf{w}}(0) = {\boldsymbol 0}.
\end{cases}
\end{equation*}
Since \(\widehat{\mathbf{w}}(0) = {\boldsymbol 0}\), we have
\begin{eqnarray*}
    |\widehat{\mathbf{w}}(\xi, t)| &\leq&  \int_0^t e^{(t - s) M_{\alpha, \mu}(\xi)} \:\!\frac{|G(\xi, s)|}{\sqrt{1 + \alpha |\xi|^{2}\,}} \,ds \\ \noalign{\smallskip}
    &\leq&  \int_0^t |G(\xi, s)| \,ds,
\end{eqnarray*}
where
$$M_{\alpha, \mu}(\xi):= - \frac{\mu|\xi|^2}{1+\alpha|\xi|^2}.$$
Thus, it follows from Lemma \ref{Lema1} that
\begin{equation}\label{imp_120824}
 \bigl|\widehat{\mathbf{w}}(\xi, t)\bigr| \leq C \!\left(|\xi| + |\xi|^{2}\right)\! \int_{0}^{t} \bigl\|\mathbf{u}(s)\bigr\|^{2}_{H^1_\alpha (\mathbb{R}^3)} \,ds, \quad \,\,\forall \; t \geq 0 \,\,\text{ and } \,\,\xi \in \mathbb{R}^{3},   
\end{equation}
where \(C > 0\) depends only on \(\alpha\). On the other hand, it follows from \eqref{diffe} that:
\[
\frac{1}{2} \frac{d}{dt} \!\left(\|\mathbf{w}(t)\|_{H^1_\alpha(\mathbb{R}^3)}^{2}\right) + \mu \|\nabla \mathbf{w}(t)\|_{L^2}^2 = \langle \operatorname{rot} (\mathbf{u} - \alpha \Delta \mathbf{u}) \times \mathbf{u}, \overline{\mathbf{u}} \:\!\rangle_{L^2}.
\]
Since
\[
\operatorname{rot} (\mathbf{u} - \alpha \Delta \mathbf{u}) \times \mathbf{u} = (\mathbf{u} \cdot \nabla) \mathbf{u} - \nabla \!\left( \frac{|\mathbf{u}|^2}{2} \right) - \alpha \operatorname{rot} (\Delta \mathbf{u}) \times \mathbf{u},
\]
we have, due to the incompressibility conditions $\operatorname{div} \mathbf{u} = \operatorname{div} \overline{\mathbf{u}} = 0$, that
\begin{align*}
   \langle (\mathbf{u} \cdot \nabla) \mathbf{u}, \overline{\mathbf{u}} \:\!\rangle_{L^2} & = -\langle (\mathbf{u} \cdot \nabla) \overline{\mathbf{u}}, \mathbf{u} \rangle_{L^2}  \leq  \|\nabla \:\!\overline{\mathbf{u}}(t)\|_{L^\infty} \| \mathbf{u}(t)\|_{L^2}^2  \\ & \leq \|\nabla \:\!\overline{\mathbf{u}}(t)\|^{\frac{1}{4}}_{L^2} \|D^3 \:\!\overline{\mathbf{u}}(t)\|^{\frac{3}{4}}_{L^2} \|\mathbf{u}(t)\|_{L^2}^2 \\ & 
   \leq \|\nabla \:\!\overline{\mathbf{u}}(t)\|^{\frac{1}{4}}_{H^1_\alpha(\mathbb{R}^3)} \|D^3 \:\!\overline{\mathbf{u}}(t)\|^{\frac{3}{4}}_{H^1_\alpha(\mathbb{R}^3)} \|\mathbf{u}(t)\|_{H^1_\alpha(\mathbb{R}^3)}^2,
\end{align*}
also that
\begin{displaymath}
- \left\langle \nabla \!\left( \frac{|\mathbf{u}|^2}{2} \right), \overline{\mathbf{u}} \right\rangle_{L^2}  =  \left\langle \frac{|\mathbf{u}|^2}{2}, \text{div} \,\overline{\mathbf{u}} \right\rangle_{L^2} = 0,
\end{displaymath}
and
\begin{align*}
- \alpha \langle \operatorname{rot}(\Delta \mathbf{u}) \times \mathbf{u}, \overline{\mathbf{u}} \:\!\rangle_{L^2} & = \alpha \langle \mathbf{u} \times \operatorname{rot}(\Delta \mathbf{u}), \overline{\mathbf{u}} \:\!\rangle_{L^2}  = \alpha \langle \:\! \overline{\mathbf{u}} , \mathbf{u} \times \operatorname{rot}(\Delta \mathbf{u}) \rangle_{L^2} \\ &  = \alpha \langle \:\! \overline{\mathbf{u}} \times \mathbf{u} , \operatorname{rot}(\Delta \mathbf{u}) \rangle_{L^2}   = \alpha \langle \operatorname{rot}(\overline{\mathbf{u}} \times \mathbf{u}), \Delta \mathbf{u} \rangle_{L^2}  
\\ & = \alpha \langle (\mathbf{u} \cdot \nabla) \overline{\mathbf{u}} - (\overline{\mathbf{u}} \cdot \nabla) \mathbf{u}, \Delta \mathbf{u} \rangle_{L^2}  
\\ & \leq \alpha \left(2 \|\nabla \:\!\overline{\mathbf{u}} \|_{L^\infty} \|\nabla \mathbf{u}\|_{L^2}^2 + \|D^2 \:\! \overline{\mathbf{u}}\|_{L^\infty} \|\mathbf{u}\|_{L^2} \|\nabla \mathbf{u}\|_{L^2} \right) 
\\ & \leq 2 \|\nabla \:\!\overline{\mathbf{u}} \|_{L^\infty} \|\mathbf{u}\|_{H_\alpha^{1}(\mathbb{R}^3)}^2 + \sqrt{\alpha\:\!} \:\! \|D^2 \:\! \overline{\mathbf{u}}\|_{L^\infty} \|\mathbf{u}\|_{H_\alpha^{1}(\mathbb{R}^3)}^{2}. 
\end{align*}
Since
$$
\| \nabla \:\! \overline{\mathbf{u}} \|_{L^\infty} \leq \| \nabla \:\! \overline{\mathbf{u}} \|_{L^2}^{\frac{1}{4}} \| D^3 \:\! \overline{\mathbf{u}} \|_{L^2}^{\frac{3}{4}},$$
and
$$ 
\| D^2 \:\! \overline{\mathbf{u}}  \|_{L^\infty} \leq \| D^2 \:\! \overline{\mathbf{u}}  \|_{L^2}^{\frac{1}{4}} \| D^4 \:\! \overline{\mathbf{u}}  \|_{L^2}^{\frac{3}{4}},
$$
it follows that
$$
\| \nabla \:\! \overline{\mathbf{u}}  \|_{L^\infty} \leq \| \nabla \:\! \overline{\mathbf{u}}  \|_{H^1_\alpha (\mathbb{R}^{3})}^{\frac{1}{4}} \| D^3 \:\! \overline{\mathbf{u}}  \|_{H^1_\alpha (\mathbb{R}^{3})}^{\frac{3}{4}}
$$
and
$$ 
\sqrt{\alpha\:\!} \:\! \| D^2 \:\! \overline{\mathbf{u}} \|_{L^\infty} \leq \| \nabla \:\! \overline{\mathbf{u}}  \|_{H^1_\alpha (\mathbb{R}^{3})}^{\frac{1}{4}} \| D^3 \:\! \overline{\mathbf{u}}  \|_{H^1_\alpha (\mathbb{R}^{3})}^{\frac{3}{4}}.
$$
From these estimates, it follows that
$$
- \alpha \langle \operatorname{rot}(\Delta \mathbf{u}) \times \mathbf{u}, \overline{\mathbf{u}} \rangle_{L^2} \leq 3 \| \nabla \:\! \overline{\mathbf{u}}  \|_{H^1_\alpha (\mathbb{R}^{3})}^{\frac{1}{4}} \| D^3 \:\! \overline{\mathbf{u}}  \|_{H^1_\alpha (\mathbb{R}^{3})}^{\frac{3}{4}} \| \mathbf{u} \|_{H^1_\alpha (\mathbb{R}^{3})}^{2},
$$
and therefore
$$
\frac{d}{dt} \!\left(\|\mathbf{w}(t)\|_{H^1_\alpha(\mathbb{R}^3)}^{2}\right) + 2 \mu \|\nabla \mathbf{w}(t)\|_{L^2}^2 \leq 8  \| \nabla \:\! \overline{\mathbf{u}}  \|_{H^1_\alpha (\mathbb{R}^{3})}^{\frac{1}{4}} \| D^3 \:\! \overline{\mathbf{u}}  \|_{H^1_\alpha (\mathbb{R}^{3})}^{\frac{3}{4}} \| \mathbf{u} \|_{H^1_\alpha (\mathbb{R}^{3})}^{2}.
$$
Now, notice that, due to Main Theorem \ref{Main1-Intro} and Proposition \ref{teoder128}, we obtain
$$
\frac{d}{dt} \!\left(\|\mathbf{w}(t)\|_{H^1_\alpha(\mathbb{R}^3)}^{2}\right) \leq - 2 \mu \|\nabla \mathbf{w}(t)\|_{L^2}^2 + H(t),
$$ 
where
$$
H(t) := C_1(t + C_2)^{- \min \left\{ \frac{7}{2} + \frac{3}{2} r^{*}, \frac{9}{2} + \frac{1}{2} r^{*} \right\}}.
$$
Multiplying the above inequality by $g(t)$ and working as before, we get
\[
\frac{d}{dt} \!\left[ g(t) \|\mathbf{w}(t)\|_{H^1_\alpha(\mathbb{R}^3)}^{2} \right] \leq (2\pi)^{-3} g'(t) \int_{B(t)} (1 + \alpha|\xi|^2) |\widehat{\mathbf{w}}(\xi, t)|^2 \, d\xi + g(t) H(t).
\]
From \eqref{imp_120824}, we infer
\[
\int_{B(t)} (1 + \alpha|\xi|^2) |\widehat{\mathbf{w}}(\xi, t)|^2 \,d \xi \leq C \!\left[ \rho(t)^5 + \rho(t)^7 + \rho(t)^9 \right] \!\left( \int_0^t  \| \mathbf{u} (s)\|_{H^1_\alpha (\mathbb{R}^{3})}^{2} \, ds \right)^{2}\!.
\]
Taking \( g(t) = C_2^{-m} (t + C_2)^m \), with
\[
C_2 \geq 1+ \frac{\alpha m}{2 \mu} , \quad \text{ and } \quad m > \max \!\left\{ \frac{5}{2} + \frac{3}{2}r^{*}, \frac{7}{2} + \frac{1}{2}r^{*}, \frac{5}{2} \right\}\!,
\]
it follows that
\[
\rho(t) = \left(\frac{m}{2 \mu}\right)^{\frac{1}{2}} \!(t + \widetilde{C}_2)^{- \frac{1}{2}},
\]
where \( \widetilde{C}_2 = C_2 - \displaystyle\frac{\alpha \:\!m}{2\mu} \geq 1 \). Thus,
\begin{eqnarray*}
\frac{d}{dt} \!\left[ g(t) \|\mathbf{w}(t)\|_{H^1_\alpha(\mathbb{R}^3)}^{2} \right] &\leq& C(t +C_2)^{m-\frac{7}{2}} \left( \int_0^t \| \mathbf{u} (s)\|_{H^1_\alpha (\mathbb{R}^{3})}^{2} \, ds \right)^2     \\ \noalign{\smallskip} & & +\, C(t + C_2)^{m - \min\left\{\frac{7}{2} + \frac{3}{2}r^{*}, \frac{9}{2} + \frac{1}{2}r^{*}\right\}}.
\end{eqnarray*}
Integrating with respect to time, we find
\begin{align*}
\|\mathbf{w}(t)\|_{H^1_\alpha(\mathbb{R}^3)}^{2} 
& \leq C(t + C_2)^{- \frac{5}{2}} \left(\int_0^t \| \mathbf{u} (\tau)\|_{H^1_\alpha (\mathbb{R}^{3})}^{2} \,d\tau  \right)^2 \\ & + C(t + C_2)^{- \min\left\{\frac{5}{2} + \frac{3}{2}r^*, \frac{7}{2} + \frac{1}{2}r^*\right\}}.
\end{align*}
Next, keeping  Theorem \ref{Main1-Intro} in mind, we will divide the study into three cases:\newline\newline 
\noindent (i) If $-\frac{3}{2} < r^* < -\frac{1}{2}$, then
\[
    \left(\int_0^t \| \mathbf{u} (\tau)\|_{H^1_\alpha (\mathbb{R}^{3})}^{2} \,d\tau  \right)^2 
    \leq C(t + C_2)^{-(1 + 2 r^*)}.
\]
\noindent (ii) If $r^* = -\frac{1}{2}$, then
\[
    \left(\int_0^t \| \mathbf{u} (\tau)\|_{H^1_\alpha (\mathbb{R}^{3})}^{2} \,d\tau  \right)^2 
    \leq C_1^2 \ln^2(t + C_2), \quad \text{ since } \quad C_2 \geq 1.
\]
\noindent (iii) If $r^* > -\frac{1}{2}$, then
\[
    \left(\int_0^t \| \mathbf{u} (\tau)\|_{H^1_\alpha (\mathbb{R}^{3})}^{2} \,d\tau  \right)^2 
    \leq C.
\]
Therefore:\newline\newline
\noindent (i) If $-\frac{3}{2} < r^* < -\frac{1}{2}$, we then have
\begin{eqnarray*}
 \|\mathbf{w}(t)\|_{H^1_\alpha(\mathbb{R}^3)}^{2} 
    &\leq&  C(t + C_2)^{-\frac{5}{2}}(t + C_2)^{-(1 + 2r^*)}  \\ \noalign{\smallskip} & + &  C(t + C_2)^{-\min\left\{\frac{5}{2} + \frac{3}{2}r^*, \frac{7}{2} + \frac{1}{2}r^*\right\}}    \\ \noalign{\smallskip} &\leq& C(t + C_2)^{-\left(\frac{7}{2} + 2r^*\right)} + C(t + C_2)^{-\min\left\{\frac{5}{2} + \frac{3}{2}r^*, \frac{7}{2} + \frac{1}{2}r^*\right\}} \\ \noalign{\smallskip} &\leq& C(t + C_2)^{-\min\left\{\frac{5}{2} + \frac{3}{2}r^*, \frac{7}{2} + \frac{1}{2}r^*\right\}} 
    \\ \noalign{\smallskip} &\leq& C(t + C_2)^{-\left(\frac{5}{2} + \frac{3}{2}r^*\right)}. 
\end{eqnarray*}    
    
\noindent (ii) If $r^* = -\frac{1}{2}$, we then have
\begin{eqnarray*}    
\|\mathbf{w}(t)\|_{H^1_\alpha(\mathbb{R}^3)}^{2} 
    &\leq& C(t + C_2)^{-\frac{5}{2}} \ln^2(t + C_2) + C(t + C_2)^{-\min\left\{\frac{5}{2} + \frac{3}{2}r^*, \frac{7}{2} + \frac{1}{2}r^*\right\}}  \\ \noalign{\smallskip} &\leq& C(t + C_2)^{-\frac{5}{2}} \ln^2(t + C_2) + C(t + C_2)^{-\frac{7}{4}}  \\ \noalign{\smallskip} &\leq& C(t + C_2)^{-\frac{7}{4}}  \\ \noalign{\smallskip} &=& C(t + C_2)^{-\left(\frac{5}{2} + \frac{3}{2}r^*\right)}.
\end{eqnarray*}

\noindent (iii) If $r^* > -\frac{1}{2}$, we then have
\begin{eqnarray*}     
    \|\mathbf{w}(t)\|_{H^1_\alpha(\mathbb{R}^3)}^{2} 
    &\leq& C(t + C_2)^{-\frac{5}{2}} + C(t + C_2)^{-\min\left\{\frac{5}{2} + \frac{3}{2}r^*, \frac{7}{2} + \frac{1}{2}r^*\right\}} \\ \noalign{\smallskip} &\leq& C(t + C_2)^{-\min\left\{\frac{5}{2} + \frac{3}{2}r^*, \frac{7}{2} + \frac{1}{2}r^*, \frac{5}{2}\right\}} \\ \noalign{\smallskip} &\leq& C(t + C_2)^{-\min\left\{\frac{5}{2} + \frac{3}{2}r^*, \frac{5}{2}\right\}}. 
\end{eqnarray*}
Combining all these cases, we obtain the desired result. 
\end{proof}
To finish, let us prove Corollary \ref{lower-bound}.
\begin{proof}[Proof of Corollary \ref{lower-bound}]
	We just need to observe that, by using the triangular inequality, we get 
	$$ 
		\|\mathbf{{u}}(t) \|_{L^2} \geq \|\mathbf{\bar{u}}(t)\|_{L^2} - \|\mathbf{{u}}(t) - \mathbf{\bar{u}}(t)\|_{L^2}. 
	$$
Now, we can conclude by observing that by Theorem \ref{Main1-Intro} and Theorem \ref{Main2-intro} lead to the desired lower bound only when the decay of linear part is slower than that of the difference.

\end{proof}

\end{document}